\documentclass[twoside,reqno]{amsart}

\numberwithin{equation}{section}
\usepackage{amsmath}
\usepackage{amssymb,amsfonts,amsthm,latexsym}
\usepackage[all]{xy}
\usepackage{hyperref}

\theoremstyle{plain}
        \newtheorem{theorem}{Theorem}[section]
        \newtheorem{lemma}[theorem]{Lemma}
        \newtheorem{proposition}[theorem]{Proposition}
        \newtheorem{corollary}[theorem]{Corollary}
               
\theoremstyle{definition}
        \newtheorem{definition}[theorem]{Definition}
        \newtheorem{remark}[theorem]{Remark}
        \newtheorem{example}[theorem]{Example}

\newcommand{\R}{\mathbb{R}}
\newcommand{\Q}{\mathbb{Q}}
\newcommand{\rat}{\mathbb{Q}}
\newcommand{\Z}{\mathbb{Z}}
\newcommand{\calC}{\mathcal{C}}

\newcommand{\calH}{\mathcal{H}}
\newcommand{\calF}{\mathcal{F}}
\newcommand{\calG}{\mathcal{G}}
\newcommand{\calI}{\mathcal{I}}
\newcommand{\ob}{\mathcal{O}}
\newcommand{\calU}{\mathcal{U}}

\newcommand {\codim} {\operatorname{codim}}
\newcommand{\ov}{\overline}
\newcommand{\im}{\mathrm{im}\hspace{1.5pt}}
\newcommand{\Hom}{\mathrm{Hom}}
\newcommand{\Homeo}{\mathrm{Homeo}}

\newcommand{\ft}[2]{\mathrm{ft}_{<#2} #1}
\newcommand{\cft}[2]{\mathrm{ft}_{\geq#2} #1}

\newcommand{\rH}[2]{H_{#2}\left(#1\right)}
\newcommand{\IH}[2]{I^{\ov{p}}H_{#2}\left(#1\right)}
\newcommand{\ICH}[2]{I_{\ov{p}}H^{#2}\left(#1\right)}
\newcommand{\ImH}[2]{IH_{#2}\left(#1\right)}

\newcommand{\IqH}[2]{I^{\ov{q}}H_{#2}\left(#1\right)}

\newcommand{\rrH}[2]{\widetilde{H}_{#2}\left(#1\right)}
\newcommand{\rrC}[2]{\widetilde{H}^{#2}\left(#1\right)}
\newcommand{\tr}[2]{#1_{< #2}}
\newcommand{\ctr}[2]{#1_{\geq #2}}
\newcommand{\HDim}[1]{\mathrm{Hdim} \left( #1\right)}
\newcommand{\rCH}[2]{H^{#2}\left(#1\right)}

\newcommand {\id} {\operatorname{id}}
\newcommand {\real} {\mathbb{R}}
\newcommand {\redH} {\widetilde{H}}
\newcommand {\smlhf} {\ensuremath{\mbox{$\frac{1}{2}$}}}
\newcommand {\Gr} {\operatorname{Gr}}
\newcommand {\pt} {\operatorname{pt}}
\newcommand {\EZ} {\operatorname{EZ}}
\newcommand {\tB} {\widetilde{B}}
\newcommand {\cplx} {\mathbb{C}}
\newcommand {\cp} {\mathbb{CP}}
\newcommand {\ftr} {\operatorname{ft}}
\newcommand {\rk} {\operatorname{rk}}
\newcommand {\intg} {\mathbb{Z}}
\newcommand {\SO} {\operatorname{SO}}
\newcommand {\ip}    {\ensuremath{I^{\overline{p}}}}
\newcommand {\iq}    {\ensuremath{I^{\overline{q}}}}

\begin{document}
\title[]{Intersection Spaces, Equivariant Moore Approximation and the Signature}

\author{Markus Banagl}
\address{Mathematisches Institut, Ruprecht-Karls-Universit\"at Heidelberg, Im Neuenheimer Feld 205,  69120 Heidelberg, Germany }
\email{banagl@mathi.uni-heidelberg.de} 
\thanks{The first author was in part supported by a research grant of the
 Deutsche Forschungsgemeinschaft.}

\author{Bryce Chriestenson}
\address{Department of Mathematics, Western Oregon University, 
  Monmouth OR 97361, USA}
\email{chriestensonb@mail.wou.edu} 
\date{\today}


\begin{abstract}
We generalize the first author's construction of intersection spaces to the case of 
stratified pseudomanifolds of stratification depth $1$ 
with twisted link bundles, assuming that each link 
possesses an equivariant Moore approximation for a suitable choice of structure group.
As a by-product, we find new characteristic classes for fiber bundles admitting
such approximations. For trivial bundles and flat bundles whose base has finite fundamental
group these classes vanish. For oriented closed pseudomanifolds, we 
prove that the reduced rational cohomology of the intersection spaces 
satisfies global Poincar\'e duality across complementary perversities if the characteristic classes vanish.
The signature of the intersection spaces agrees with the Novikov signature of
the top stratum.
As an application, these methods yield new results
about the Goresky-MacPherson intersection homology signature of pseudomanifolds. 
We discuss several nontrivial examples, such as the case of flat bundles and symplectic toric manifolds.  
\end{abstract}
\maketitle
\tableofcontents
\bibliographystyle{amsplain}

\section{Introduction}

Classical approaches to Poincar\'e duality on singular spaces are Cheeger's
$L^2$ cohomology with respect to suitable conical metrics on the regular part of the space
(\cite{cheeger1}, \cite{cheeger2}, \cite{cheeger3}),
and Goresky-MacPherson's intersection homology \cite{gm1}, \cite{gm2}, 
depending on a perversity parameter $\bar{p}$.
More recently, the first author has introduced and investigated a different, spatial perspective
on Poincar\'e duality for singular spaces (\cite{BA}).
This approach associates to certain classes of singular spaces $X$ a cell complex
$\ip X$, which depends on a perversity $\bar{p}$ and is called an \emph{intersection space}
of $X$. Intersection spaces are required to be generalized rational geometric Poincar\'e complexes in the
sense that when $X$ is a closed oriented pseudomanifold, there is a Poincar\'e duality isomorphism
$\redH^{i} (\ip X;\rat) \cong \redH_{n-i} (\iq X;\rat)$, where
$n$ is the dimension of $X$, $\bar{p}$ and $\bar{q}$ are complementary perversities in the sense
of intersection homology theory, and $\redH^*, \redH_*$ denote reduced singular (or cellular)
cohomology and homology.

The resulting homology theory $HI_*^{\bar{p}} (X) = H_* (\ip X;\rat)$ and cohomology theory
$HI^*_{\bar{p}} (X) = H^* (\ip X;\rat)$ 
is \emph{not} isomorphic to intersection (co)homology 
$I^{\bar{p}} H_* (X;\rat),$ $I_{\bar{p}} H^* (X;\rat).$ Since its inception, the 
theory $HI^*_{\bar{p}}$ has so far had applications in areas ranging from 
fiber bundle theory and computation of equivariant cohomology (\cite{BAGGD}),
K-theory (\cite[Chapter 2.8]{BA}, \cite{SP}),
algebraic geometry
(smooth deformation of singular varieties (\cite{bmjta}, \cite{bmjs}), perverse sheaves \cite{bbm},
mirror symmetry \cite[Chapter 3.8]{BA}),
to theoretical Physics (\cite[Chapter 3]{BA}, \cite{bbm}).
For example, the approach of intersection spaces makes it straightforward
to define intersection $K$-groups by $K^* (\ip X)$.
These techniques are not accessible to classical intersection cohomology.
There are applications to $L^2$-theory as well:
In \cite{bahu}, for every perversity $\bar{p}$ 
a Hodge theoretic description of the theory $\redH I^*_{\bar{p}} (X;\real)$ is found;
that is, a Riemannian metric on the top stratum (which is in fact a fiberwise scattering metric and thus 
very different from Cheeger's class of
metrics) and a suitable space of $L^2$ harmonic forms with respect to this metric
(the extended weighted $L^2$ harmonic forms for suitable weights) which is
isomorphic to $\redH I^*_{\bar{p}} (X;\real)$.
A de Rham description of $HI^*_{\bar{p}} (X;\real)$ has been given in \cite{Ba-new}
for two-strata spaces whose link bundle is flat with respect to the isometry group 
of the link.

At present, intersection spaces have been constructed for isolated singularities and for
spaces with stratification depth $1$ whose link bundles are a global product, \cite{BA}.
Constructions of $\ip X$ in some depth $2$ situations have been provided in
\cite{bandepth2}. The fundamental idea in all of these constructions
is to replace singularity links by their Moore approximations, a concept 
from homotopy theory Eckmann-Hilton dual to the concept of Postnikov approximations.
In the present paper, we undertake a systematic treatment of twisted link bundles.
Our method is to employ equivariant Moore approximations of links with respect to the
action of a suitable structure group for the link bundle.

Equivariant Moore approximations are introduced in Section \ref{sec.equivmooreapprox}.
On the one hand, the existence of such approximations is obstructed and we
give a discussion of some obstructions. For instance, if $S^{n-1}$ is the fiber sphere of a
linear oriented sphere bundle, then the structure group can be reduced so as to allow
for an equivariant Moore approximation to $S^{n-1}$ of degree $k$, $0<k<n,$
if and only if the Euler class of the sphere bundle
vanishes (Proposition \ref{eulercondition}). If the action of a group $G$ on a space $X$
allows for a $G$-equivariant map $X\to G$, then the existence of a $G$-equivariant
Moore approximation to $X$ of positive degree $k$ implies that the rational homological
dimension of $G$ is at most $k-1$.
On the other hand, we present geometric situations where equivariant Moore approximations
exist.
 If the group acts trivially on a simply connected
CW complex $X$, then a Moore approximation of $X$ exists.
If the group acts cellularly and the cellular boundary operator in degree $k$ vanishes or
is injective, then $X$ has an equivariant Moore approximation.
Furthermore, equivariant Moore approximations exist often for the effective
Hamiltonian torus action of a symplectic toric manifold. For instance, we prove
(Proposition \ref{prop.toricmfddim4}) that $4$-dimensional symplectic toric manifolds
always possess $T^2$-equivariant Moore approximations of any degree.

In Section \ref{fiberwise}, we use equivariant Moore approximations to
construct fiberwise homology truncation and cotruncation. Throughout, we use
homotopy pushouts and review their properties (universal mapping property,
Mayer-Vietoris sequence) in Section \ref{pushout}.
Proposition \ref{relationtoIH} relates the homology of fiberwise (co)truncations
to the intersection homology of the cone bundle of the given bundle.
Of fundamental importance for the later developments is Lemma \ref{lem.flkcgkexact},
which shows how the homology of the total space of a bundle is built up from
the homology of the fiberwise truncation and cotruncation.
In order to prove these facts, we employ a notion of precosheaves together with
an associated local to global technique explained in Section \ref{sec.localtoglobal}.
Proposition \ref{bundle duality} establishes Poincar\'e duality between fiberwise
truncations and complementary fiberwise cotruncations. 

At this point, we discover a new set of \emph{characteristic classes}
\[ \ob_i (\pi,k,l) \subset H^d (E;\rat),~ d=\dim E,~ i=0,1,2,\ldots, \]
defined for fiber bundles $\pi: E\to B$ which
possess degree $k,l$ fiberwise truncations (Definition \ref{def.locdualobs}).
We show that these characteristic classes
vanish if the bundle is a global product (Proposition \ref{prop.globalprodobvan}).
Furthermore, they vanish for flat bundles if the fundamental group of the base is finite
(Theorem \ref{thm.obvanflatfinitefund}). On the other hand, we construct in
Example \ref{expl.obnonzero} a bundle $\pi$ for which $\ob_2 (\pi,2,1)$ does not vanish.
The example shows also that the characteristic classes $\ob_*$ seem to be rather subtle,
since the bundle of the example is such that all the differentials of its Serre
spectral sequence do vanish.

Now the relevance of these characteristic classes vis-\`a-vis Poincar\'e duality is
the following: While, as mentioned above, there is always a Poincar\'e duality
isomorphism between truncation and complementary cotruncation,
this isomorphism is not determined uniquely and may not commute with 
Poincar\'e duality on the given total space $E$.
Proposition \ref{prop.dob0} states that the duality isomorphism in degree $r$ between
fiberwise truncation and cotruncation can be chosen to commute with
Poincar\'e duality on $E$ if and only if $\ob_r (\pi,k,l)$ vanishes. In this case,
the duality isomorphism is uniquely determined by the commutation requirement.
Thus, we refer to the classes $\ob_*$ as \emph{local duality obstructions}, since in the subsequent
application to singular spaces, these classes are localized at the singularities.

The above bundle-theoretic analysis is then applied in Section \ref{sec.intspacespd} 
in constructing intersection
spaces $\ip X$ for Thom-Mather stratified pseudomanifolds of stratification depth 1,
whose link bundles allow for structure groups with equivariant Moore approximations.
The central definition is \ref{IX};
the requisite Thom-Mather theory is reviewed in Section \ref{sec.thommather}.
The main result here, Theorem \ref{JXPD}, establishes generalized Poincar\'e duality
\begin{equation} \label{equ.pdisoipiq}
\redH^r (\ip X;\rat)\cong \redH_{n-r} (\iq X;\rat) 
\end{equation}
for complementary perversity intersection spaces,
provided the local duality obstructions of the link bundle vanish.

In the Sections \ref{sec.ihsign},
\ref{sec.signintsp}, we investigate the signature and Witt element of intersection forms.
We show first that if a Witt space allows for middle-degree equivariant Moore approximation,
then its intersection form on intersection homology agrees with the intersection form
of the top stratum as an element in the Witt group $W(\rat)$ of the rationals
(Corollary \ref{IHsignature}). Section \ref{sec.signintsp} shows that the duality isomorphism
(\ref{equ.pdisoipiq}), where we now use the (lower) middle perversity,
 can in fact be constructed so that the associated middle-degree intersection form
is symmetric when the dimension $n$ is a multiple of $4$. Let $\sigma (IX)$ denote the
signature of this symmetric form. 
Theorem \ref{HIsignature} asserts that $\sigma (IX)=\sigma (M,\partial M),$
where $\sigma (M,\partial M)$ denotes the signature of the top stratum.
In particular then, $\sigma (IX)$ agrees with the intersection homology signature.
For the rather involved proof of this theorem, we build
on the method of Spiegel \cite{SP}, which in turn is partially based
on the methods introduced in the proof of \cite[Theorem 2.28]{BA}.
It follows from all of this that there are interesting global signature obstructions
to fiberwise homology truncation in bundles. For instance,
viewing complex projective space $\cp^2$ as a stratified space with bottom stratum
$\cp^1 \subset \cp^2,$ the signature of $\cp^2$ is $1$, whereas the signature
of the top stratum $D^4$ vanishes. Indeed, the normal circle bundle of $\cp^1,$
i.e. the Hopf bundle, does not have a degree $1$ fiberwise homology truncation,
as can of course be verified directly.

\emph{On notation}: Throughout this paper, all homology and cohomology groups are taken with rational coefficients.
Reduced homology and cohomology will be denoted by $\widetilde{H}_*$ and $\widetilde{H}^*$.
The linear dual of a $K$-vector space $V$ is denoted by $V^\dagger = \Hom (V,K)$.

\section{Properties of Homotopy Pushouts}
\label{pushout}

This paper uses homotopy pushouts in many constructions. We recall here
their definition, as well as the two properties we will need: their universal mapping property and
the associated Mayer-Vietoris sequence.
\begin{definition}\label{pushoutdef}
Given continuous maps $\xymatrix{ Y_1 & X \ar[l]_{f_1} \ar[r]^{f_2} & Y_2}$ between topological spaces we define the 
\emph{homotopy pushout} of $f_1$ and $f_2$ to be the topological space $Y_1 \cup_X Y_2$, the quotient 
of the disjoint union $X\times [0,1] \sqcup Y_1 \sqcup Y_2$ by the smallest equivalence relation generated by  
$$ \left\{\left(x,0\right)\sim f_1\left(x\right)|~ x\in X\right\}\cup\left\{ \left(x,1\right) \sim f_2\left(x\right) |~ x \in X\right\}$$
We denote $\xi_i : Y_i \rightarrow Y_1 \cup_X Y_2$, for $i=1,2$, and 
$\xi_0:X\times I \rightarrow Y_1 \cup_X Y_2$, to be the inclusions into the disjoint union followed by the quotient map,
where $I=[0,1]$.
\end{definition}

\begin{remark}\label{pushoutuniversal}
The homotopy pushout satisfies the following universal mapping property:
Given any topological space $Z$, continuous maps $g_i : Y_i \rightarrow Z$, and homotopy 
  $h:X\times I \rightarrow Z$ satisfying    $h\left(x,i\right) = g_{i+1}\circ f_{i+1}\left(x\right)$ for $x\in X$, and $i=0,1$, 
  then there exists a unique continuous map $g: Y_1 \cup_X Y_2 \rightarrow Z$ such that $g_i = g \circ \xi_i$ for $i=1,2,$
  and $h=g\circ \xi_0$.\\

From the data of a homotopy pushout we get a long exact sequence of homology groups
	\begin{equation}\label{pushoutMV}\xymatrix{ \cdots \ar[r] & \rH{X}{r} \ar[r]^(.375){\left(f_{1 *}, f_{2 *}\right)} & \rH{Y_1}{r}\oplus \rH{Y_2}{r} \ar[r]^(.55){\xi_{1 *} - \xi_{2 *}} & \rH{Y_1 \cup_X Y_2}{r} \ar[r]^(.625)\delta & \cdots}\end{equation}
This is the usual Mayer-Vietoris sequence applied to $Y_1 \cup_X Y_2$ when it is decomposed into the union of 
$(Y_1 \cup_X Y_2) \setminus Y_i$ for $i=1,2$, whose overlap is $X$ crossed with the open interval. 
If $X$ is not empty, then there is also a version for reduced homology:
\begin{equation} \label{pushoutMVred}
\xymatrix{ 
\cdots \ar[r] & \rrH{X}{r} \ar[r]^(.375){\left(f_{1 *}, f_{2 *}\right)} & 
\rrH{Y_1}{r}\oplus \rrH{Y_2}{r} \ar[r]^(.55){\xi_{1 *} - \xi_{2 *}} & \rrH{Y_1 \cup_X Y_2}{r} \ar[r]^(.625)\delta & \cdots}
\end{equation}
\end{remark}

\section{Equivariant Moore Approximation}
\label{sec.equivmooreapprox}

Our method to construct intersection spaces for twisted link bundles rests on the concept
of an equivariant Moore approximation. The transformation group of the general abstract
concept will eventually be a suitable reduction of the structure group of a fiber bundle,
which will enable fiberwise truncation and cotruncation. The basic idea behind degree-$k$
Moore approximations of a space $X$ is to find a space $X_{<k}$, whose homology agrees
with that of $X$ below degree $k$, and vanishes in all other degrees. It is well-known that
Moore-approximations cannot be made functorial on the category of all topological spaces and
continuous maps, as explained in \cite{BA}. The equivariant Moore space problem
was raised in 1960 by Steenrod, who asked whether given a group $G$, a right
$\intg[G]$-module $M$ and an integer $k>1$, there exists a topological space $X$ such that
$\pi_1 (X)=G,$ $H_i (\widetilde{X};\intg)=0,$ $i\not= 0,k,$
$H_0 (\widetilde{X};\intg)=\intg,$ and $H_k (\widetilde{X};\intg)=M,$ where
$\widetilde{X}$ is the universal cover of $X$, equipped with the $G$-action by covering translations.
The first counterexample was due to Gunnar Carlsson, \cite{carlsson}.
Further work on Steenrod's problem has been done by Douglas Anderson,
James Arnold, Peter Kahn, Frank Quinn, and Justin Smith.

\begin{definition}\label{Gspace}
Let $G$ be a topological group.  A \emph{$G$-space} is a pair 
$\left(X,\rho_X\right)$, where $X$ is a topological space and 
$\rho_X:G \rightarrow \Homeo\left(X\right)$ is a continuous group homomorphism. 
A \emph{morphism} between $G$-spaces $f:\left(X,\rho_X\right)\rightarrow \left(Y,\rho_Y\right)$ 
is a continuous map $f:X\rightarrow Y$ that satisfies 
\[ \rho_Y (g)\circ f = f \circ \rho_X (g), \text{ for every } g \in G. \]
We denote the set of morphisms by $\Hom_G (X,Y)$.
Morphisms are also called \emph{$G$-equivariant maps}.
We will write $g\cdot x = \rho_X (g)(x),$ $x\in X,$ $g\in G$.
\end{definition}
Let $cX$ be the closed cone $X\times [0,1] / X\times \left\{0\right\}$.  
If $X$ is a $G$-space, then the cone $cX$ becomes a $G$-space in a natural way: the cone point
is a fixed point and for $t\in (0,1],$
$g\in G$ acts by $g\cdot (x,t) = (g\cdot x,t)$.
More generally, 
given $G$-equivariant maps $\xymatrix{ Y_1 & X \ar[l]_{f_1} \ar[r]^{f_2} & Y_2},$
the homotopy pushout $Y_1 \cup_X Y_2$ is a $G$-space in a natural way.

\begin{definition}\label{EqMooreApx}
Given a $G$-space $X$ and an integer $k\geq 0$, a \emph{$G$-equivariant Moore approximation} 
to $X$ of degree $k$ is a $G$-space $\tr{X}{k}$ together with a continuous $G$-equivariant map 
$\tr{f}{k}: \tr{X}{k}\rightarrow X$, satisfying the following properties:
\begin{itemize}
\item $\rH{\tr{f}{k}}{r} : \rH{\tr{X}{k}}{r}\rightarrow \rH{X}{r}$ is an isomorphism for all $r<k$, and
\item $\rH{\tr{X}{k}}{r} = 0$ for all $r\geq k$.
\end{itemize}
\end{definition}

\begin{definition}
Let $X$ be a nonempty topological space.  The ($\mathbb{Q}$-coefficient) \emph{homological dimension} of $X$ is the number 
$$\HDim{X} = \min\left\{ n \in \Z : \mbox{ $\rH{X}{m} = 0$ for all $m > n$} \right\},$$
if such an $n$ exists.
If no such $n$ exists, then we say that $X$ has infinite homological dimension.
\end{definition}

\begin{example}\label{trivialexamples}
There are two extreme cases, in which equivariant Moore approximations are trivial to construct.
For $k=0$, any Moore approximation must satisfy $\rH{\tr{X}{0}}{i} = 0$, for all $i\geq 0$.  
This forces $X_{<0} = \emptyset$, and $\tr{f}{0}$ is the empty function.  If $X$ has $\HDim{X} = n$, then for $k\geq n+1$ set $\tr{X}{k} = X$ and $\tr{f}{k} = \id_X$.  Hence, any space of homological dimension $n$ has an equivariant Moore approximation of degrees $k\leq 0$ and $k>n$.
\end{example}

\begin{example}
If $G$ acts trivially on a simply connected CW complex $X$, then Moore approximations of $X$ exist in
every degree.
For spatial homology truncation in the nonequivariant case, see Chapter 1 of \cite{BA}, which also
contains a discussion of functoriality issues arising in connection with Moore approximations. The simple connectivity
condition is sufficient, but far from necessary.
\end{example}

\begin{example}
Let $G$ be a compact Lie group acting smoothly on a smooth manifold $X$.
Then, according to \cite{IL},
one can arrange a CW structure on $X$ in such a way that $G$ acts cellularly.
Now suppose that $X$ is any $G$-space equipped with a CW structure such that $G$ acts cellularly.
If the $k$-th boundary operator $\partial_k: C_k (X)\to C_{k-1} (X)$ in the cellular chain complex of 
$X$ vanishes, then the $(k-1)$-skeleton of $X$, together with its inclusion into $X$ and endowed 
with the restricted $G$-action, is an
equivariant Moore-approximation $\tr{X}{k} = X^{k-1}$.
This condition is for example satisfied for the standard minimal CW structure on complex projective
spaces and tori. However, in order to make a given action cellular, one may of course be forced
to endow spaces with larger, nonminimal, CW structures.
Similarly,
if $\partial_k$ is injective, then $\tr{X}{k} = X^{k}$
is an equivariant Moore-approximation.
\end{example}

The following observation can sometimes be used to show that certain $G$-spaces and degrees do not
allow for an equivariant Moore approximation.
\begin{proposition}\label{homobstruction}
Let $G$ be a topological group and $X$ a nonempty $G$-space.  Let $G_\lambda$ be the $G$-space $G$ 
with the action by left translation.  If 
$$\Hom_G \left( X , G_\lambda \right) \neq \emptyset$$
and $X$ has a $G$-equivariant Moore approximation of degree $k>0$, then
$$k-1 \geq \HDim{G}.$$
\end{proposition}
\begin{proof}
Let $\tr{f}{k}: \tr{X}{k} \to X$ be a $G$-equivariant Moore approximation, $k>0$.
Precomposition with $\tr{f}{k}$ induces a map 
$$\tr{f}{k}^\sharp : \Hom_G\left(X,G_\lambda\right)\rightarrow \Hom_G\left(\tr{X}{k},G_\lambda\right).$$
As $k>0$ and $X$ is not empty, we have
$H_0 (\tr{X}{k})\cong H_0 (X)\not= 0$. Thus $\tr{X}{k}$ is not empty.
For each $\phi \in \Hom_G \left(\tr{X}{k},G_\lambda\right)$, we note that $\phi$ is surjective 
since $\tr{X}{k}$ is not empty, 
left translation is transitive and $\phi$ is equivariant.  Choose $x\in \tr{X}{k}$ such that $\phi\left(x\right)= e $.  Define $h_x : G \rightarrow \tr{X}{k}$ by $h_x\left(g\right)= g\cdot x$.  Then $\phi\circ h_x = \id_G$, since 
	$$\phi\left(h_x\left(g\right)\right)= \phi\left(g\cdot x\right) = g\phi\left(x\right) = ge = g.$$
Therefore the map induced by $\phi$ on homology has a splitting induced by $h_x$, so there is an isomorphism
	$$\rH{\tr{X}{k}}{r}\cong A_r \oplus \rH{G}{r}$$
for some subgroup $A_r\subset \rH{\tr{X}{k}}{r}$ and every $r$.  Since by definition $\rH{\tr{X}{k}}{r}=0$ for $r\geq k$, 
then if such a $\phi$ exists we must have $\HDim{G} \leq k-1$.  
The condition $\Hom_G\left(X,G_\lambda\right) \neq \emptyset$ is sufficient to guarantee the existence of such a $\phi$.
\end{proof}

\begin{example}
By Proposition \ref{homobstruction}, the action of $S^1$ on itself by rotation does 
not have an equivariant Moore space approximation of degree $1$.

Consider $S^1$ acting on $X=S^1\times S^2$ by rotation in the first coordinate and trivially in the second coordinate. 
Example \ref{trivialexamples} shows that for $k\leq 0$ and $k\geq 4$, $S^1$-equivariant Moore approximations
exist trivially. By Proposition \ref{homobstruction}, there is no such approximation for $k=1$. 
We shall now construct an approximation for degree $k=2$.
Fix a point $y_0\in S^2$.  let $i:S^1\rightarrow X,$ $\theta \mapsto \left(\theta,y_0\right),$ be the inclusion at $y_0$.  
Let $S^1$ act on itself by rotation, then the map $i$ is equivariant.  
Furthermore, by the K\"unneth theorem we know that $\rH{X}{1}\cong \Q$ is generated 
by the class $[S^1 \times y_0]$, and $\rH{i}{1}:\rH{S^1}{1}\rightarrow \rH{X}{1}$ is an isomorphism 
taking $[S^1]$ to $[S^1\times y_0]$.  Thus since both $S^1$ and $X$ are connected, we have that the map $i$ gives a $S^1$-equivariant Moore space approximation of degree $2$.
\end{example}

Further positive results asserting the existence of Moore approximations in geometric situations
such as symplectic toric manifolds are discussed in Section \ref{sec.sphsymptorman}.

\section{Precosheaves and Local to Global Techniques}
\label{sec.localtoglobal}

The material of this section is fairly standard; 
we include it in order to fix terminology and notation.
Let $B$ be a topological space and let $VS_{\Q}$ denote the category of rational vector spaces and linear maps.
\begin{definition}
A covariant functor $\calF: \tau B \rightarrow VS_\Q$ from the category $\tau B$ of open sets on $B$, 
with inclusions for morphisms, to the category $VS_\Q$, is called a \emph{precosheaf} on $B$.  
For open sets $U\subset V\subset B$ we denote the result of applying $\calF$ to the inclusion map $U\subset V$ 
by $i_{U,V}^{\calF}:\calF\left(U\right)\rightarrow \calF\left(V\right)$.
A \emph{morphism} $f:\calF \rightarrow \calG$ of precosheaves on $B$ is a natural transformation of functors.
\end{definition}

Let $\calU=\left\{ U_\alpha \right\}_{\alpha \in \Lambda}$ be an open cover of $B$, and let $\tau\calU$ be the 
category whose objects are unions of finite intersections of open sets in $\calU$ and 
whose morphisms are inclusions.  There is a natural inclusion functor $u:\tau\calU \rightarrow \tau B,$ 
regarding an open set in $\tau \calU$ as an object of $\tau B$.  
This realizes $\tau \calU$ as a full subcategory of $\tau B$.

\begin{definition}\label{Ulocallyconstant}
A precosheaf $\calF$ on $B$ is \emph{$\calU$-locally constant} if for any $U_\alpha \in \calU$ and any $U$ 
which is a finite intersection of elements of $\calU$ and intersects $U_\alpha$ nontrivially, the map 
$$i_{U_\alpha \cap U, U_\alpha}^\calF:\calF\left(U_\alpha  \cap U \right) \rightarrow \calF\left(U_\alpha\right)$$
is an isomorphism.
\end{definition}

Consider the product category $\tau \calU \times \tau \calU$ whose objects are 
pairs $\left(U,V\right)$ with $U,V \in \tau \calU$, and whose morphism are pairs of 
inclusions $\left(U,V\right)\rightarrow \left(U',V'\right)$ given by $U\subset U'$ and $V\subset V'$.  
Define the functors $\cap,\cup : \tau \calU \times \tau \calU \rightarrow \tau \calU$ that take the 
object $\left(U,V\right)$ to  $U\cap V$ and $U\cup V$, respectively, and the morphism 
$\left(U,V\right)\rightarrow \left(U',V'\right)$ to the inclusions $U\cap V \subset U'\cap V'$ and 
$U\cup V \subset U'\cup V'$.  Similarly we have the projection functors 
$p_i:\tau \calU \times \tau \calU \rightarrow \tau \calU$, for $i=1,2$ where $p_i$ projects 
onto the $i$-th factor.  The inclusions $U,V \subset U\cup V$ and $U\cap V \subset U,V$ 
induce natural transformations of functors $j_i : p_i \to \cup$, and $\iota_i:\cap \to p_i$ for $i=1,2$.  
Applying a precosheaf $\calF$ to the $j_i (U,V)$, we obtain linear maps
$\calF (U)\to \calF (U\cup V),$ $\calF (V)\to \calF (U\cup V),$ which we will again denote
by $j_1, j_2$ (rather than $\calF (j_i (U,V))$). Similarly for the $\iota_i$.
Thus for any precosheaf $\calF$ on $B$ we have the morphisms 
$$\xymatrix{ \calF\left(U\cap V\right)\ar[r]^{\left(\iota_1,\iota_2\right)}& 
\calF\left(U\right)\oplus\calF\left(V\right)\ar[r]^{j_1-j_2} & \calF\left(U\cup V\right)}$$
for any object $\left(U,V\right)$ in $\tau \calU \times \tau \calU$.
The functoriality of $\calF$ implies 
that $\left(j_1 - j_2\right)\circ \left(\iota_1,\iota_2\right) = 0$.

Any morphism of precosheaves $f:\calF \rightarrow \calG$ gives a commutative diagram
\begin{equation} \label{f commutes}
\xymatrix{ \calF\left(U\cap V\right)\ar[r]^{\left(\iota_1,\iota_2\right)} \ar[d]^{f\left(U\cap V\right)}& \calF\left(U\right)\oplus\calF\left(V\right)\ar[r]^{j_1-j_2} \ar[d]^{f\left(U\right)\oplus f\left(V\right)}& 
\calF\left(U\cup V\right) \ar[d]^{f\left(U\cup V\right)}\\ \calG\left(U\cap V\right)\ar[r]^{\left(\iota_1,\iota_2\right)}& \calG\left(U\right)\oplus\calG\left(V\right)\ar[r]^{j_1-j_2} & \calG\left(U\cup V\right).}
\end{equation}

\begin{definition}\label{UMV}
Let $\calF_r$ be a collection of precosheaves on $B$, for $r\geq 0$, and let $\calU$ be an open cover of $B$.  We say that the sequence $\calF_r$ satisfies the \emph{$\calU$-Mayer-Vietoris property} 
if there is a natural transformation of functors on $\tau\calU\times \tau \calU$, 
$$\delta_i^\calF : \calF_{i} \circ \cup \longrightarrow \calF_{i-1}\circ \cap,$$
for each $i$, such that for every pair of open sets $U,V \in \tau \calU$ the following sequence is exact:
$$\xymatrix@C=20pt{
 \ar[r] & \calF_{i+1}\left(U\cup V\right) \ar[r]^{\delta_{i+1}^{\calF}} & 
\calF_i\left(U\cap V\right) \ar[r]^(.45){\left(\iota_1^i,\iota_2^i\right)} & 
\calF_i \left(U\right) \oplus \calF_i \left(V\right) \ar[r]^(.55){j_1^i-j_2^i} & 
\calF_i\left(U\cup V\right) \ar[r]^(.65){\delta_i^{\calF}} &. }$$
A collection of morphisms $f_r :\calF_r \rightarrow \calG_r$, for $r\geq 0$, is called \emph{$\delta$-compatible} 
if for each pair of open sets $U,V\in \tau \calU$ the following diagram commutes for all $i\geq 0$:
\begin{equation}\label{f compatible}\xymatrix@C=50pt{
 \calF_{i+1}\left(U\cup V\right) \ar[r]^(.525){\delta_{i+1}^{\calF}\left(U,V\right)} 
  \ar[d]_{f_{i+1}\left(U\cup V\right)} & 
  \calF_i\left(U\cap V\right)  \ar[d]^{f_{i}\left(U\cap V\right)} \\   
\calG_{i+1}\left(U\cup V\right) \ar[r]^(.525){\delta_{i+1}^{\calG}\left(U,V\right)} & \calG_i\left(U\cap V\right). } 
\end{equation}
\end{definition}

\begin{proposition}\label{local to global}
Let $B$ be a compact topological space and
let $\calU$ be an open cover of $B$.  Let $f_i : \calF_i \rightarrow \calG_i$ be a sequence of $\delta$-compatible morphisms between $\calU$-locally constant precosheaves on $B$ that satisfy the $\calU$-Mayer-Vietoris property.  If $f_i\left(U\right):\calF_i\left(U\right)\rightarrow \calG_i\left(U\right)$ is an isomorphism for every $U \in \calU$ and for every $i\geq 0$, then $f_i\left(B\right):\calF_i\left(B\right)\rightarrow \calG_i\left(B\right)$ is an isomorphism for all $i\geq 0$.
\end{proposition}

\begin{proof}
We shall prove the following statement by induction on $n$:
For every $U\in \tau \calU$ which can be written as a union
$U= U_1 \cup \cdots \cup U_n$ of $n$ open sets $U_j \in \tau \calU,$
each of which is a finite intersection of open sets in $\calU,$
the map $f_i (U): \calF_i (U)\to \calG_i (U)$ is an isomorphism for all $i\geq 0$.
The base case $n=1$ follows from the fact that $\calF_i, \calG_i$ are $\calU$-locally constant together with
the assumption on $f_i (U)$ for $U\in \calU$.  
Denote $U^j = U_1\cup \cdots \cup \hat{U_j}\cup \cdots \cup U_n$ and 
$V^j = \left(U_1 \cap U_j \right) \cdots \cup \hat{U_j}\cup \cdots \cup \left( U_n \cap U_j\right)$; then 
$U=U^j \cup U_j$ and $V^j = U^j \cap U_j$.  Since the $f_i$ are $\delta$-compatible, by 
\eqref{f compatible} and \eqref{f commutes} we have the commutative diagram below, whose rows are the $\calU$-Mayer-Vietoris sequences associated to the pair $U^j$ and $U_j$:
	$$\xymatrix{ \ar[r] & \calF_i \left( V^j\right) \ar[r] \ar[d]^{f_i\left(V^j\right)}& 
  \calF_i\left(U^j\right)\oplus\calF_i\left(U_j\right) \ar[r] \ar[d]^{f_i\left(U^j\right)\oplus f_i\left(U_j\right)} & 
  \calF_i\left(U\right) \ar[r]^\delta \ar[d]^{f_i\left(U\right)} & 
  \calF_{i-1}\left(V^j\right) \ar[r]\ar[d]^{f_{i-1}\left(V^j\right)} & \\ 
 \ar[r] & \calG_i \left( V^j\right) \ar[r] & \calG_i\left(U^j\right)\oplus\calG_i\left(U_j\right) \ar[r] & 
  \calG_i\left(U\right) \ar[r]^\delta &\calG_{i-1}\left(V^j\right) \ar[r] &.}$$
Each of $V^j, U^j,$ and $U_j$ is a union of less than $n$ open sets, each of which is a finite intersection of elements of $\calU$.  
Thus by induction hypothesis, $f_i (V^j),$ $f_i (U^j)$ and $f_i (U_j)$ are isomorphisms for all $i$.
By the $5$-lemma, $f_i (U)$ is an isomorphism for all $i$, which concludes the induction step.
Since $B$ is compact, there is a finite number of open sets in $\calU$ which cover $B$. Thus the induction yields the desired result.
\end{proof}

\section{Examples of Precosheaves}

Throughout this section we consider a fiber bundle $\pi:E\rightarrow B$ with fiber $L$ and 
topological structure group $G$.  
We assume that $B,E$, and $L$ are compact oriented manifolds such that $E$ is compatibly oriented with respect
to the orientation of $B$ and $L$.  Set $n=\dim E$, $b=\dim B$ and $c=\dim L= n-b$.  We may form the fiberwise 
cone of this bundle, $DE$, by defining $DE$ to be the homotopy pushout, Definition \ref{pushoutdef}, of the pair of maps 
$\xymatrix{ B & E \ar[l]_{\pi} \ar[r]^{\id} & E.}$  By Remark \ref{pushoutuniversal}, the map $\pi$ induces a map 
$\pi_D : DE \rightarrow B$, given by $\id_B$ on $B$ and $(x,t)\mapsto \pi (x)$ for
$(x,t)\in E\times I$.  
This makes $DE$ into a fiber bundle whose fiber is $cL$, the cone on $L$, and whose 
structure group is $G$.  We point out, for $U\subset B$ open, that $\pi_D^{-1} U \rightarrow U$ is obtained as the 
homotopy pushout of the pair of maps $\xymatrix{ U & \pi^{-1} U \ar[l]_{\pi|_{\pi^{-1} U}} \ar[r]^{\id} & \pi^{-1} U }$.  
One more fact 
that will be needed is that the pair  $\left(DE,E\right)$, where $E$ is identified with $E\times \left\{1\right\} \subset DE$, 
along with a stratification of $DE$ given by $B\subset DE,$ is a compact $\Q$-oriented $\partial$-stratified 
topological pseudomanifold, 
in the sense of Friedman and McClure \cite{FM}.  Here we have identified $B$ with the image $\sigma\left(B\right)$ of the 
``zero section'' $\sigma:B\rightarrow DE,$ 
sending $x\in B$ to the cone point of $cL$ over $x$.
Similarly for any open $U\subset B,$ the 
pair $\left(\pi^{-1}_D U, \pi^{-1} U\right)$ is a $\Q$-oriented $\partial$-stratified 
pseudomanifold, though it will not be compact 
unless $U$ is compact. We write $\partial \pi^{-1}_D U = \pi^{-1} U.$

\begin{example}\label{singularhomology}
For each $r\geq 0,$ there are precosheaves $\pi_*\calH_r$ on $B$ defined by
	$$U\mapsto \rH{\pi^{-1}\left(U\right)}{r}.$$
By the Eilenberg-Steenrod axioms, these are $\calU$-locally constant, and satisfy the $\calU$-Mayer-Vietoris 
property for any good open cover $\calU$ of $B$.  
(An open cover $\calU$ of a $b$-dimensional manifold is \emph{good}, if every nonempty finite intersection of sets in
$\calU$ is homeomorphic to $\real^b$. Such a cover exists if the manifold is smooth or PL.)

Let $\pi':E' \rightarrow B$ be another fiber bundle, and $f:E \rightarrow E'$ a morphism of fiber bundles.  Then $f$ induces a morphism of precosheaves $f_* : \pi_* \calH_r \rightarrow \pi'_* \calH_r$, given on any open set $U\subset B$ by
$$f_* (U) := \left( f|_{\pi^{-1}U}\right)_*: 
    \rH{\pi^{-1}U}{r}\rightarrow \rH{\pi'^{-1}U}{r}.$$
Furthermore, for any pair of open sets $U,V \subset B,$ we have the following commutative diagram whose rows are exact Mayer-Vietoris sequences:
\begin{equation}\label{Bundle MV} 
\xymatrix@C=14pt{ \ar[r] &\rH{\pi^{-1} (U\cap V)}{r} \ar[r] \ar[d]^{f_r\left(U\cap V\right)}& 
\rH{\pi^{-1}U}{r}\oplus \rH{\pi^{-1}V}{r} \ar[r] \ar[d]^{f_r\left(U\right)\oplus f_r\left(V\right)}& 
\rH{\pi^{-1}(U\cup V)}{r} \ar[r]^>>>>\delta \ar[d]^{f_r\left(U\cup V\right)}& \\
\ar[r] & \rH{\pi'^{-1} (U\cap V)}{r} \ar[r] & 
\rH{\pi'^{-1}U}{r}\oplus \rH{\pi'^{-1}V}{r} \ar[r] & \rH{\pi'^{-1}(U\cup V)}{r} \ar[r]^>>>>\delta & }
\end{equation}
Thus, for any good open cover $\calU,$ the map $f$ induces a $\delta$-compatible sequence of morphisms between 
precosheaves which satisfy the $\calU$-Mayer-Vietoris property, and are $\calU$-locally constant.
\end{example}

\begin{example}\label{intersectionhomology}
Define the precosheaf of intersection homology groups, $\pi_{D*} \calI^{\ov{p}} \calH_r$ for each $r\geq0$, and 
each perversity $\ov{p}$, by assigning to the open set $U\subset B$ the vector space, $\IH{\pi_{D}^{-1} U}{r}$. 
We use the definition of intersection homology via finite 
singular chains as in \cite{FM}.  This is a slightly more general definition than that of King,\cite{K}, 
and Kirwan-Woolf \cite{KW}.  For our situation the definitions all agree with the exception 
that the former allows for more 
general perversities, see the comment after Prop. 2.3 in \cite{FM} for more details.  In Section $4.6$ of 
Kirwan-Woolf \cite{KW} it is shown that each $\pi_{D*} \calI^{\ov{p}} \calH_r$ is a precosheaf for each $r\geq 0$, and 
that this sequence satisfies the $\calU$-Mayer-Vietoris property for any open cover $\calU$ of $B$.  Furthermore, these 
are all $\calU$-locally constant for any good cover $\calU$ of $B$.  

Let $f:E \to E'$ be a bundle morphism with $\dim E \geq \dim E'$.
Using the levelwise map $E\times I \to E\times I,$ $(e,t)\mapsto (f(e),t),$ and the identity map on $B$,
$f$ induces a bundle morphism $f_D: DE \to DE'$. Recall that a continuous map between stratified spaces
is called \emph{stratum-preserving} if the image of every pure stratum of the source 
is contained in a single pure stratum of the target. A stratum-preserving map $g$ is called \emph{placid}
if $\codim g^{-1} (S)\geq \codim S$ for every pure stratum $S$ of the target. Placid maps induce covariantly
linear maps on intersection homology (which is not true for arbitrary continuous maps).
The map $f_D$ is indeed stratum-preserving and, since $\dim E \geq \dim E'$, placid and thus
induces maps
\[ (f_D|_{\pi^{-1}_D (U)})_*: \IH{\pi_{D}^{-1} U}{r} \longrightarrow \IH{{\pi'_D}^{-1} U}{r} \]
for each open set $U\subset B$.
This way, we obtain a sequence of $\delta$-compatible morphisms 
$f_{D*}: \pi_{D*} \calI^{\ov{p}} \calH_r \rightarrow \pi'_{D*} \calI^{\ov{p}} \calH_r$.

With $I^{\ov{p}} C_* (X)$ the singular rational intersection chain complex as in \cite{FM},
we define intersection cochains by $I_{\ov{p}} C^* (X) = \Hom (I^{\ov{p}} C_* (X), \rat)$ and 
intersection cohomology by $I_{\ov{p}} H^* (X) = H^* (I_{\ov{p}} C^* (X))$. Then the universal coefficient
theorem
\[ I_{\ov{p}} H^* (X) \cong \Hom(I^{\ov{p}} H_* (X),\rat) \]
holds.
Theorem 7.10 of \cite{FM} establishes Poincar\'e-Lefschetz duality for compact $\Q$-oriented $n$-dimensional $\partial$-stratified pseudomanifolds $\left(X,\partial X\right)$. Some important facts are established there in the proof:
\begin{enumerate}
\item For complementary perversities $\ov{p}+\ov{q} = \ov{t}$, there is a commutative diagram whose rows are exact:
\begin{equation}\label{boundary PD}
\xymatrix{ \ar[r]^<<<<<<{j_{\partial}^r}  & 
\ICH{X}{r} \ar[r]^{i_{\partial}^r} \ar[d]^{D_{X}^r}_{\cong} & 
\ICH{\partial X}{r} \ar[r]^{\delta^r_{\partial}} \ar[d]^{D^r_{\partial X}}_{\cong} & 
\ICH{X,\partial X}{r+1} \ar[r] \ar[d]^{D_{n-r-1}^{X}}_{\cong} & \\
					 \ar[r]^<<<<<<{j^{\partial}_{n-r}} & \IqH{X,\partial X}{n-r} \ar[r]^{\delta^{\partial}_{n-r}} & 
\IqH{\partial X}{n-r-1} \ar[r]^{i_{n-r-1}^{\partial}} &	\IqH{X}{n-r-1} \ar[r] & }
\end{equation}

\item The inclusion $X\setminus \partial X\rightarrow X$ induces an isomorphism
	\begin{equation}\label{openclosed} \IqH{X\setminus \partial X}{n-r} \cong \IqH{X}{n-r}.\end{equation}
\end{enumerate}

Consider the smooth oriented $c$-dimensional manifold $L$.  
The closed cone $cL$ is a compact $\Q$-oriented $(c+1)$-dimensional $\partial$-stratified pseudomanifold.  
Thus the long exact sequence coming from the bottom row of diagram \eqref{boundary PD} gives
\begin{equation}\label{closed cone}
\xymatrix{ \ar[r] & \IH{cL,L}{r+1}\ar[r]^(.575){\delta_{r+1}^{\partial}} & \IH{L}{r} \ar[r]^{i^{\partial}_r} & 
\IH{cL}{r} \ar[r]^(.45){j_r^{\partial}} & \IH{cL,L}{r} \ar[r] & }.
\end{equation}
\end{example}

\begin{proposition}\label{Cone Formula}
Let $\ov{p}$ be a perversity and let $k = c - \ov{p}\left(c+1\right)$.  
Then for the maps in the exact sequence \eqref{closed cone} we have an isomorphism
	$$i^{\partial}_r :\rH{L}{r} \rightarrow \IH{cL}{r}\!,$$
when $r<k$, and an isomorphism
	$$\delta^{\partial}_{r+1} : \IH{cL,L}{r+1} \rightarrow \rH{L}{r}\!,$$
when $r\geq k$.
\end{proposition}
\begin{proof}
The standard cone formula for intersection homology asserts that for a closed
$c$-dimensional manifold $L$, the inclusion $L\hookrightarrow cL$ as the boundary induces an isomorphism
$\IH{L}{r} \cong \IH{cL}{r}$ for $r<c-\ov{p}(c+1)$, whereas
$\IH{cL}{r}=0$ for $r\geq c-\ov{p}(c+1).$
(By (\ref{openclosed}) above, this holds both for the closed and the open cone.)
This already establishes the first claim. The second one follows from the cone formula together
with the exact sequence (\ref{closed cone}).
\end{proof}

\section{Fiberwise Truncation and Cotruncation}
\label{fiberwise}

Let $\pi: E\rightarrow B$ be a fiber bundle of closed manifolds with closed manifold fiber $L$ and structure group $G$
such that $B,E$ and $L$ are compatibly oriented.  
Suppose that a $G$-equivariant Moore approximation $L_{<k}$ of degree $k$ exists for the fiber $L$.
The bundle $E$ has an underlying principal $G$-bundle $E_P \to B$
such that $E=E_P \times_G L$. Using the $G$-action on $\tr{L}{k}$, we set
\[ \ft{E}{k} = E_P \times_G \tr{L}{k}. \]
Then $\ft{E}{k}$ is the total space of a fiber bundle
$\tr{\pi}{k}:\ft{E}{k}\rightarrow B$ with fiber $\tr{L}{k},$ structure group $G$ and
underlying principal bundle $E_P$.
The equivariant structure map $\tr{f}{k}:\tr{L}{k} \to L$ defines a morphism of bundles 
\[ \tr{F}{k}:\ft{E}{k}=E_P \times_G \tr{L}{k}\rightarrow E_P \times_G L =E. \]
\begin{definition}
The pair $\left(\ft{E}{k},\tr{F}{k}\right)$ is called the \emph{fiberwise $k$-truncation} of the bundle $E$. 
\end{definition}

\begin{definition}\label{cotruncationdef}
The \emph{fiberwise $k$-cotruncation} $\cft{E}{k}$ is the homotopy pushout of the pair of maps 
$$\xymatrix{ B & \ft{E}{k} \ar[l]_{\tr{\pi}{k}} \ar[r]^{\tr{F}{k}} & E}.$$
Let $\ctr{c}{k}:E\rightarrow \cft{E}{k}$, and $\sigma:B\rightarrow \cft{E}{k}$ be the maps $\xi_2$ and $\xi_1,$ respectively, 
appearing in Definition \ref{pushoutdef}.
\end{definition}

Since $\tr{F}{k}$ satisfies $\tr{\pi}{k} = \pi \circ \tr{F}{k}$ we have, 
by the universal property of Remark \ref{pushoutuniversal}, using the constant homotopy, 
a unique map $\ctr{\pi}{k}:\cft{E}{k}\rightarrow B$ satisfying 
$\pi = \ctr{\pi}{k}\circ \ctr{c}{k}$, $\ctr{\pi}{k} \circ \sigma = \id_B$
and $(\pi_{\geq k} \circ \xi_0)(x,t)=\pi_{<k} (x)$ for all $t\in I$, where
$\xi_0: \ft{E}{k} \times I \to \cft{E}{k}$ is induced by the inclusion
(as in Definition \ref{pushoutdef}).
The map $\ctr{\pi}{k}:\cft{E}{k}\rightarrow B$ is a fiber bundle projection 
with fiber the homotopy pushout of
$$\xymatrix{ \star & L_{<k} \ar[l] \ar[r]^{\tr{f}{k}} & L},$$
i.e. the mapping cone of $\tr{f}{k}$.
Note that this mapping cone is a $G$-space in a natural way (with $\star$
as a fixed point), since $\tr{f}{k}$ is equivariant.
The map
$\ctr{c}{k}:E\rightarrow \cft{E}{k}$ is a morphism of fiber bundles.  
Furthermore, the bundle $\pi_{\geq k}$ has a canonical section $\sigma$, sending
$x\in B$ to $\star$ over $x$.

\begin{definition}\label{PandQdef}
Define the space $\ctr{Q}{k}E$ to be the homotopy pushout of the pair of maps 
	$$\xymatrix{ \star & B \ar[l] \ar[r]^{\sigma} & \cft{E}{k} }.$$
This is the mapping cone of $\sigma$ and hence
\[ \redH_* (\ctr{Q}{k}E) \cong H_* (\cft{E}{k},B), \]
where we identified $B$ with its image under 
the embedding $\sigma$.
Define the maps $\ctr{\xi}{k}:\cft{E}{k}\rightarrow \ctr{Q}{k}E$ and 
$[c] : \star \rightarrow \ctr{Q}{k}E$ to be the maps $\xi_2$ and $\xi_1,$ respectively 
(Definition \ref{pushoutdef}). Set
$$C_{\geq k} = \ctr{\xi}{k}\circ \ctr{c}{k}: E \rightarrow \ctr{Q}{k}E.$$
\end{definition}

For each open set $U\subset B,$ the space $\ctr{\pi}{k}^{-1}U$ is the pushout of the 
pair of maps $\xymatrix{ U & \pi_{<k}^{-1} U \ar[l]_{\pi_{<k}|} \ar[r]^{F_{<k}|} & \pi^{-1}U }$
and the restrictions of $c_{\geq k}$ induce a morphism of fiber bundles
$c_{\geq k}(U):\pi^{-1} U \to \pi^{-1}_{\geq k} U$.
Define the precosheaf $\pi_*^Q \calH_r$ by the assignment 
$U \mapsto H_r (\pi^{-1}_{\geq k} U,U)$ (again identifying $U$ with its image under $\sigma$).
That this assignment is indeed a precosheaf follows from the functoriality of homology applied to the
commutative diagram of inclusions
\[ \xymatrix{
(\pi^{-1}_{\geq k} U, U) \ar[r] \ar[rd] & (\pi^{-1}_{\geq k} V, V) \ar[d] \\
& (\pi^{-1}_{\geq k} W,W)
} \]
associated to nested open sets $U\subset V\subset W$.
The maps $C^k_r (U): H_r (\pi^{-1} U)\to H_r (\pi^{-1}_{\geq k} U,U),$ given by the composition
\[ H_r (\pi^{-1} U) \stackrel{c_{\geq k}(U)_*}{\longrightarrow}     
  H_r (\pi^{-1}_{\geq k} U) \longrightarrow H_r (\pi^{-1}_{\geq k} U,U),  \]
define a morphism of precosheaves 
	$$\calC^k_r : \pi_*\calH_r \rightarrow \pi_*^Q \calH_r$$
for all $r\geq 0$.
The following lemma justifies the terminology ``cotruncation''. 
\begin{lemma}  \label{lem.ckrforrbtimesl}
For $U\cong \real^b,$ the map $C^k_r (U)$ 
is an isomorphism for $r\geq k,$ while
$H_r (\pi^{-1}_{\geq k} U,U)=0$ for $r<k$.
\end{lemma}
\begin{proof}
Let $L_{\geq k}$ denote the mapping cone of $f_{<k}: L_{<k} \to L$.
Since the bundles $\pi$ and $\pi_{\geq k}$ both (compatibly) trivialize over $U\cong \real^b$, the
map $C^k_r (U)$ can be identified with the composition
\[ H_r (\real^b \times L) \longrightarrow
  H_r (\real^b \times L_{\geq k}) \longrightarrow H_r (\real^b \times (L_{\geq k},\star)),  \]
which can further be identified with
\[ H_r (L) \longrightarrow \redH_r (L_{\geq k}). \]
This map fits into a long exact sequence
\[ H_r (L_{<k}) \stackrel{f_{<k*}}{\longrightarrow} H_r (L)
 \longrightarrow \redH_r (L_{\geq k}) \longrightarrow H_{r-1} (L_{<k}). \]
The result then follows from the defining properties of the Moore approximation $f_{<k}$.
\end{proof}

As in Example \ref{singularhomology}, the map 
$\tr{F}{k,r}:\rH{\ft{E}{k}}{r}\rightarrow \rH{E}{r}$ is $\tr{\calF}{k,r}\left(B\right)$ for the 
morphism of precosheaves $\tr{\calF}{k,r}:\pi_{<k *} \calH_r\rightarrow \pi_*\calH_r$ 
given by $F_{<k}|_*: H_r (\pi^{-1}_{<k} U)\to H_r (\pi^{-1} U)$
for each $r\geq 0$.

For each open set $U$ we have the long exact sequence of perversity $\ov{p}$-intersection homology groups 
	\begin{equation}\label{disk local}
\xymatrix{\cdots \ar[r] & \IH{\pi^{-1}_D U,\partial \pi^{-1}_D U}{r+1} 
\ar[r]^(.675){\delta_{r+1}^{\partial}\left(U\right)} & 
\rH{\pi^{-1}U}{r}\ar[r]^(.45){i^{\partial}_r\left(U\right)} & 
\IH{\pi^{-1}_D U}{r} \ar[r]^(.675){j^{\partial}_r\left(U\right)} \ar[r] & 
\cdots}
\end{equation}
(Recall that $\pi_D: DE\to B$ is the projection of the cone bundle.)  
When $U$ varies, this exact sequence forms a
precosheaf of acyclic chain complexes.  
In particular the morphisms $i^{\partial}_r$ and $\delta^{\partial}_{r+1}$ are morphisms of precosheaves for every $r\geq 0$. From now on, in order to have good open covers, we assume that $B$ is either smooth or at least PL.

\begin{proposition}\label{relationtoIH}
Fix a perversity $\ov{p}$.  Let $n-1=\dim E$, $b=\dim B$, $c=n-b-1$, and $k=c - \ov{p}\left(c+1\right)$.  
Assume that $B$ is compact and that an equivariant Moore approximation 
$\tr{f}{k}:\tr{L}{k}\rightarrow L$ to $L$ of degree $k$ exists. Then the compositions 
	$$i^{\partial}_r\left(B\right) \circ \tr{F}{k *} : \rH{\ft{E}{k}}{r}\rightarrow \IH{DE}{r}$$
and
	$$C^k_r \circ \delta^{\partial}_{r+1}\left(B\right): \IH{DE,E}{r+1}\rightarrow 
              H_r (\cft{E}{k},B)\cong \rrH{\ctr{Q}{k}E}{r}$$
are isomorphisms for all $r\geq 0$.
\end{proposition}
\begin{proof}
We use our local to global technique.  
Let $\calU$ be a finite good open cover of $B$ which trivializes $E$. 
The map $\tr{F}{k}$ induces (by restrictions to preimages of open subsets) a map of precosheaves 
as demonstrated in Example \ref{singularhomology}. 
Both $i^{\partial}_r$ and $\tr{F}{k,*}$ are sequences of $\delta$-compatible morphisms of 
$\calU$-locally constant precosheaves that satisfy the $\calU$-Mayer-Vietoris property.  
Let $U\in \calU$, then $\rH{\tr{\pi}{k}^{-1} U}{r} \cong \rH{\tr{L}{k}}{r}$ and 
$\tr{\calF}{k,r}=\tr{f}{k*}$ is an isomorphism in degrees $r<k$ and $0$ in degrees $r\geq k$.  
Likewise by Proposition \ref{Cone Formula}, the map $i^{\partial}_r$ induces an isomorphism 
$\rH{L}{r} \cong \IH{\pi^{-1}_{D}U}{r}$ in degrees $r<k$ and $0$ in degrees $r\geq k$, since
$\pi^{-1}_D U \cong U \times cL\cong \real^b \times cL$, $\IH{\real^b \times cL}{r} \cong \IH{cL}{r}$, 
and we can identify $i^{\partial}_r\left(U\right)$ with $i^\partial_r$ from \eqref{closed cone}. 
Thus, the composition is an isomorphism in every degree.  We can now apply Proposition \ref{local to global} 
to obtain the desired result. 

A analogous argument gives the desired result for the second statement, using
Lemma \ref{lem.ckrforrbtimesl} in conjunction with Proposition \ref{Cone Formula} to establish
the base case.
\end{proof}

It follows from Proposition \ref{relationtoIH} that 
$i^\partial_r (B): H_r (E)\to \IH{DE}{r}$ is surjective for all $r$,
$F_{<k*}: H_r (\ft{E}{k})\to H_r (E)$ is injective for all $r$,
$C^k_r: H_r (E)\to H_r (\cft{E}{k},B)$ is surjective for all $r$, and
$\delta^\partial_{r+1} (B): \IH{DE,E}{r+1} \to H_r (E)$ is injective for all $r$.
We may use the isomorphisms in Proposition \ref{relationtoIH} to identify 
$\rH{\ft{E}{k}}{r}$ with $\IH{DE}{r}$
and $\rrH{\ctr{Q}{k}E}{r}$ with $\IH{DE,E}{r+1}\!$.  In doing so, we may consider the exact sequence 
	\begin{equation}\label{DE,E}\xymatrix{ \ar[r]&\IH{DE,E}{r+1} 
\ar[r]^(.65){\delta^{\partial}_{r+1}} & \rH{E}{r} \ar[r]^(.4){i^{\partial}_r}& 
\IH{DE}{r} \ar[r]^(.65){j^{\partial}_r} & },\end{equation}
and identify $\tr{F}{k,r}$ as a section of $i^{\partial}_r$, and $C^k_r$ as a 
section of $\delta^{\partial}_{r+1}$.  Thus we see that $j^\partial_r = 0$ for every $r\geq 0$, 
and we have a split short exact sequence
	\begin{equation}\label{splitexact}\xymatrix{ 0 \ar[r] & 
\IH{DE,E}{r+1} \ar@/^/[r]^(.625){\delta^{\partial}_{r+1}} & 
\rH{E}{r} \ar@/^/[r]^(.475){i^\partial_r} \ar@/^/[l]^(.375){C^k_r}& 
\IH{DE}{r} \ar[r] \ar@/^/[l]^(.525){\tr{F}{k,r}}& 0.}\end{equation}

\begin{lemma} \label{lem.flkcgkexact}
The sequence
\[ 0\rightarrow 
H_r (\ft{E}{k}) \stackrel{F_{<k,*}}{\longrightarrow}
H_r (E) \stackrel{C_{\geq k,*}}{\longrightarrow}
\redH_r (Q_{\geq k} E)
\rightarrow 0
\]
is exact.
\end{lemma}
\begin{proof}
Only exactness in the middle remains to be shown.
The standard sequence
\[ \ft{E}{k} \stackrel{F_{<k}}{\longrightarrow} E
 \hookrightarrow \operatorname{cone} (F_{<k}) \]
induces an exact sequence
\begin{equation} \label{equ.coneflk}
H_r (\ft{E}{k}) \stackrel{F_{<k,r}}{\longrightarrow} H_r (E)
 \longrightarrow \redH_r (\operatorname{cone} (F_{<k})). 
\end{equation}
Collapsing appropriate cones yields homotopy equivalences
\[ \operatorname{cone}(F_{<k}) \stackrel{\simeq}{\longrightarrow}
 \cft{E}{k}/B \stackrel{\simeq}{\longleftarrow} Q_{\geq k} E \]
such that the diagram
\[ \xymatrix{
E \ar@{^{(}->}[r] \ar@{^{(}->}[d]_{c_{\geq k}} & 
    \operatorname{cone}(F_{<k}) \ar[r]^{\simeq} & \cft{E}{k}/B \ar@{=}[d] \\
\cft{E}{k} \ar@{^{(}->}[r]^{\xi_{\geq k}} & Q_{\geq k} E \ar[r]^{\simeq} & \cft{E}{k}/B
} \]
commutes (even on the nose, not just up to homotopy).
The induced diagram on homology,
\[ \xymatrix{
H_r (E) \ar[r] \ar[d]_{c_{\geq k*}} & 
    \redH_r (\operatorname{cone}(F_{<k})) \ar[r]^{\cong} & \redH_r (\cft{E}{k}/B) \ar@{=}[d] \\
H_r (\cft{E}{k}) \ar[r]^{\xi_{\geq k*}} & \redH_r (Q_{\geq k} E) \ar[r]^{\cong} & \redH_r (\cft{E}{k}/B),
} \]
shows that the homology kernel of $E\to \operatorname{cone}(F_{<k})$ equals the kernel
of $\xi_{\geq k*} c_{\geq k*} = C_{\geq k*}$, but it also equals the image of $F_{<k,r}$ by
the exactness of (\ref{equ.coneflk}).
\end{proof}

\begin{proposition}\label{bundle duality}
Let $n-1=\dim E$, $b=\dim B$ and $c=n-b-1$.
For complementary perversities $\ov{p}+\ov{q}=\ov{t}$, let 
$k=c-\ov{p}\left(c+1\right)$ and $l=c-\ov{q}\left(c+1\right)$.
Assume that an 
equivariant Moore approximation to $L$ exists of degree $k$ and of degree $l$. Then there is a 
Poincar\'e duality isomorphism 
\[ D_{k,l}: H^r (\ft{E}{k}) \cong \redH_{n-r-1} (Q_{\geq l} E). \]
\end{proposition}
\begin{proof}
We use the isomorphisms in Proposition \ref{relationtoIH} 
and the Poincar\'e-Lefschetz duality of \cite{FM}, 
as described here in \eqref{boundary PD}, applied to the $\partial$-stratified 
pseudomanifold $\left(DE,E\right)$.
By definition, $D_{k,l}$ is the unique isomorphism such that
\[ \xymatrix@C=40pt{
I_{\bar{p}} H^r (DE) \ar[r]^{F_{<k}^* \circ i^*}_\cong \ar[d]_{\cong}^{D_{DE}} &
  H^r (\ft{E}{k}) \ar@{..>}[d]^{D_{k,l}} \\
I^{\bar{q}} H_{n-r} (DE,E) \ar[r]^{C^l_{n-r-1} \circ \delta}_\cong &
  \redH_{n-r-1} (Q_{\geq l} E)
} \]
commutes.
\end{proof}

It need not be true, however, that the diagram
\begin{equation}\label{goal}
		\xymatrix{ \rCH{E}{r} \ar[r]^{\tr{F}{k}^*} \ar[d]_{D_E}^{\cong} & 
  \rCH{\ft{E}{k}}{r} \ar[d]^{D_{k,l}}_{\cong} \\ 
 \rH{E}{n-r-1} \ar[r]^{\ctr{C}{l *}} & \rrH{\ctr{Q}{l}E}{n-r-1} }
	\end{equation}
commutes, see Example \ref{expl.obnonzero} below. It turns out that there is an obstruction to
the existence of any isomorphism $H^r (\ft{E}{k}) \cong \redH_{n-r-1} (Q_{\geq l} E)$ such that
the diagram (\ref{goal}) commutes.
\begin{definition} \label{def.locdualobs}
Let $k,l$ be two integers.
Given $G$-equivariant Moore approximations 
$f_{<k}: L_{<k}\to L,$ $f_{<l}: L_{<l}\to L,$ the
\emph{local duality obstruction} in degree $i$ is defined to be
\[ \ob_i (\pi, k,l) = \{ C^*_{\geq k} (x)\cup C^*_{\geq l} (y) ~|~ x\in \redH^i (Q_{\geq k} E),~
   y\in \redH^{n-1-i} (Q_{\geq l} E) \} \subset H^{n-1} (E). \]
\end{definition}
Locality of this obstruction refers to the fact that in the context of stratified spaces, the
obstruction arises only near the singularities of the space. Clearly, the definition of $\ob_i (\pi,k,l)$
does not require any smooth or PL structure on $B$ and thus is available
for topological base manifolds.
The obstruction set $\ob_i (\pi,k,l)$ is a cone: If $z =C^*_{\geq k} (x)\cup C^*_{\geq l} (y)$
is in $\ob_i (\pi,k,l)$ then for any $\lambda \in \rat,$
\[ \lambda z = C^*_{\geq k} (\lambda x)\cup C^*_{\geq l} (y) \in \ob_i (\pi,k,l). \]
If $E$ is connected, then $H^{n-1}(E)\cong \rat$ is one-dimensional, so either
$\ob_i (\pi,k,l)=0$ or $\ob_i (\pi,k,l)\cong \rat$.

\begin{proposition}  \label{prop.dob0}
There exists an isomorphism $D: H^r (\ft{E}{k}) \cong \redH_{n-r-1} (Q_{\geq l} E)$ such that
\[
		\xymatrix{ \rCH{E}{r} \ar[r]^{\tr{F}{k}^*} \ar[d]_{D_E}^{\cong} & 
  \rCH{\ft{E}{k}}{r} \ar[d]^{D}_{\cong} \\ 
 \rH{E}{n-r-1} \ar[r]^{\ctr{C}{l *}} & \rrH{\ctr{Q}{l}E}{n-r-1} }
\]
commutes if and only if the local duality obstruction $\ob_r (\pi,k,l)$ vanishes.
In this case, $D$ is uniquely determined by the diagram.
\end{proposition}
\begin{proof}
We have seen that both $F^*_{<k}$ and $C_{\geq l*}$ are surjective and their respective images
have equal rank. Thus by linear algebra
$D$ exists if and only if $D_E (\ker F^*_{<k}) =  \ker C_{\geq l*}$.
By Lemma \ref{lem.flkcgkexact}, $\ker F^*_{<k} = \im C^*_{\geq k}$.
Thus the condition translates to:
For every $x\in \redH^r (Q_{\geq k} E),$ $C_{\geq l*} D_E C^*_{\geq k} (x) =0$.
Rewriting this entirely cohomologically using the universal coefficient theorem,
this translates further to
\[ C^*_{\geq k} (x)\cup C^*_{\geq l} (y)=0 \]
for all $x,y$.

The uniqueness of $D$ is standard: If $x\in H^r (\ft{E}{k})),$ then
$D(x) = C_{\geq l*} D_E (x'),$ where $x' \in H^r (E)$ is any element with $F^*_{<k} (x')=x$.
By the condition on the kernels, this is independent of the choice of $x'$.
\end{proof}

\begin{proposition} \label{prop.obzerodisdkl}
If $\ob_i (\pi,k,l)=0,$ then the unique $D$ given by Proposition \ref{prop.dob0} equals the $D_{k,l}$
constructed in Proposition \ref{bundle duality}.
\end{proposition}
\begin{proof}
This follows from the diagram
\[ \xymatrix{
I_{\ov{p}} H^r (DE) \ar[r]^{i^*} \ar[d]_{D_{DE}}^\cong & H^r (E) \ar[r]^{F^*_{<k}} \ar[d]_{D_E}^\cong 
  & H^r (\operatorname{ft}_{<k} E) \ar[d]^D_\cong \\
I^{\ov{q}} H_{n-r} (DE,E) \ar[r]^{\delta}  & H_{n-r-1} (E) \ar[r]^{C_{\geq l*}}  
  & \redH_{n-r-1} (Q_{\geq l} E). 
} \]
The left hand square is part of the commutative ladder (\ref{boundary PD}).
The right hand square commutes by the construction of $D$.
Since the horizontal compositions are isomorphisms, $D=D_{k,l}$.
\end{proof}
Although superficially simple, this proposition has rather interesting geometric ramifications:
Since $D_{k,l}$ can \emph{always} be defined, even when the duality obstruction is not zero,
the proposition implies that in such a case, diagram (\ref{goal}) cannot commute.
This means that $D_{k,l}$ is not always a geometrically ``correct'' duality isomorphism,
and the duality obstructions govern when it is and when it is not.

It was already shown in \cite[Section 2.9]{BA} that if the link bundle is a global product, then
Poincar\'e duality holds for the corresponding intersection spaces. This suggests that the
duality obstruction vanishes for a global product. We shall now verify this directly:
\begin{proposition} \label{prop.globalprodobvan}
For complementary perversities $\ov{p}+\ov{q}=\ov{t}$, let 
$k=c-\ov{p}\left(c+1\right)$ and $l=c-\ov{q}\left(c+1\right)$.
If $\pi: E=B\times L \to B$ is a global product, then $\ob_i (\pi,k,l)=0$ for all $i$.
\end{proposition}
\begin{proof}
We have $\operatorname{ft}_{\geq k} E = B\times L_{\geq k}$ and
by the K\"unneth theorem, the reduced cohomology of $Q_{\geq k} E$ is given by
\begin{align*} 
\redH^* (Q_{\geq k} E) &= H^* (\operatorname{ft}_{\geq k} E, B) = 
  H^* (B\times L_{\geq k}, B\times \star) = H^* (B\times (L_{\geq k}, \star)) \\
  &\cong H^* (B)\otimes H^* (L_{\geq k}, \star). 
\end{align*}
Let $f_{\geq k}: L\to L_{\geq k}$ be the structural map associated to the cotruncation.
By the naturality of the cross product, the square
\[ \xymatrix{
H^* (E) & H^* (B) \otimes H^* (L) \ar[l]_\times^\cong \\
\redH^* (Q_{\geq k} E) \ar[u]^{C^*_{\geq k}} & H^* (B) \otimes H^* (L_{\geq k}, \star) \ar[l]_\times^\cong 
 \ar[u]_{\id \otimes f^*_{\geq k}} 
} \]
commutes.
Let $x\in \redH^i (Q_{\geq k} E),$ $y\in \redH^{n-1-i} (Q_{\geq l} E).$
Their images under the Eilenberg-Zilber map are of the form
\[ \EZ (x) = \sum_r b_r \otimes e^{\geq k}_r,~ b_r \in H^* (B),~ e^{\geq k}_r \in H^* (L_{\geq k},\star), \]
\[ \EZ (y) = \sum_s b'_s \otimes e^{\geq l}_s,~ b'_s \in H^* (B),~ e^{\geq l}_s \in H^* (L_{\geq l},\star), \]
$\deg b_r + \deg e^{\geq k}_r = i,$
$\deg b'_s + \deg e^{\geq l}_s = n-1-i.$
Thus
\[ (\id \otimes f^*_{\geq k})\EZ (x)\cup (\id \otimes f^*_{\geq l})\EZ (y) =
\left( \sum_r b_r \otimes f^*_{\geq k} (e^{\geq k}_r) \right) \cup
   \left( \sum_s b'_s \otimes f^*_{\geq l} (e^{\geq l}_s) \right) \]
and
\begin{align*}
C^*_{\geq k} (x)\cup C^*_{\geq l} (y) &= 
  \times \circ (\id \otimes f^*_{\geq k})\EZ (x) \cup \times \circ (\id \otimes f^*_{\geq l})\EZ (y) \\
&= \left( \sum_r b_r \times f^*_{\geq k} (e^{\geq k}_r) \right) \cup
   \left( \sum_s b'_s \times f^*_{\geq l} (e^{\geq l}_s) \right) \\
&= \sum_{r,s} \pm (b_r \cup b'_s) \times (f^*_{\geq k} (e^{\geq k}_r) \cup f^*_{\geq l} (e^{\geq l}_s)).
\end{align*}
If $\deg f^*_{\geq k} (e^{\geq k}_r) + \deg f^*_{\geq l} (e^{\geq l}_s) <\dim L,$ then
$\deg b_r + \deg b'_s >\dim B$ and thus $b_r \cup b'_s =0$.
If $\deg f^*_{\geq k} (e^{\geq k}_r) + \deg f^*_{\geq l} (e^{\geq l}_s) >\dim L,$ then
trivially $f^*_{\geq k} (e^{\geq k}_r) \cup f^*_{\geq l} (e^{\geq l}_s)=0$.
Finally, if $\deg f^*_{\geq k} (e^{\geq k}_r) + \deg f^*_{\geq l} (e^{\geq l}_s) =\dim L,$
then $f^*_{\geq k} (e^{\geq k}_r) \cup f^*_{\geq l} (e^{\geq l}_s)=0$ by the defining properties
of cotruncation and the fact that $k$ and $l$ are complementary.
This shows that
\[ C^*_{\geq k} (x)\cup C^*_{\geq l} (y)=0. \]
\end{proof}

This result means that, as for other characteristic classes, the duality obstructions
of a bundle are a measure of how twisted a bundle is.
An important special case is $\ov{p}(c+1) = \ov{q}(c+1)$. Then $k=l,$ $Q_{\geq k} E = Q_{\geq l} E$,
and for $x\in \redH^i (Q_{\geq k} E),$ $y\in \redH^{n-1-i} (Q_{\geq k} E),$
\[ C^*_{\geq k} (x)\cup C^*_{\geq l} (y) = C^*_{\geq k} (x\cup y). \]
By the injectivity of $C^*_{\geq k},$ this product vanishes if and only if $x\cup y=0$.
So in the case $k=l$ the local duality obstruction $\ob_* (\pi,k,k)$ vanishes if and only
complementary cup products in $\redH^* (Q_{\geq k} E)$ vanish. For a global product this
is indeed always the case, by Proposition \ref{prop.globalprodobvan}.
\begin{example}
Let $B=S^2,$ $L=S^3$ and $E=B\times L = S^2 \times S^3$.
Then $c=3$ and, taking $\ov{p}$ and $\ov{q}$ to be lower and upper middle perversities,
\[ k= 3-\ov{m} (4) = 2 = 3- \ov{n} (4) = l. \]
The degree $2$ Moore approximation is $L_{<2} = \pt$ and the cotruncation is
$L_{\geq 2} \simeq S^3 =L$. Thus 
\[ \operatorname{ft}_{\geq 2} E = B\times L_{\geq 2} \simeq S^2 \times S^3 =E. \]
The reduced cohomology $\redH^i (Q_{\geq 2} E) = H^i (S^2 \times (S^3, \pt))$
is isomorphic to $\rat$ for $i=3,5$ and zero for all other $i$.
Thus all (and in particular, the complementary) cup products vanish and so the local duality obstruction
$\ob_* (\pi,2,2)$ vanishes.
\end{example}
Here is an example of a fiber bundle whose duality obstruction does not vanish.
\begin{example} \label{expl.obnonzero}
Let $Dh$ be the disc bundle associated to the Hopf bundle $h:S^3 \to S^2$, i.e.
$Dh$ is the normal disc bundle of $\cplx P^1$ in $\cplx P^2$. Now take two copies 
$Dh_+ \to S^2_+$ and $Dh_- \to S^2_-$ of this disc bundle and define $E$ as the double
\[   E = Dh_+ \cup_{S^3} Dh_-. \]
Then $E$ is the fiberwise suspension of $h$ and so an $L=S^2$-bundle over $B=S^2$,
with $L$ the suspension of a circle. Let $\sigma_+, \sigma_- \in L$ be the two suspension points.
The bundle $E$ is the sphere bundle of a real $3$-plane vector bundle $\xi$ over $S^2$
with $\xi = \eta \oplus \underline{\real}^1$, where $\eta$ is the real $2$-plane bundle whose
circle bundle is the Hopf bundle and $\underline{\real}^1$ is the trivial line bundle.
The points $\sigma_\pm$ are fixed points under the
action of the structure group on $L$.
Let $\ov{p}$ be the lower, and $\ov{q}$ the upper middle perversity. Here $n=5,$ $b=2$ and $c=2$.
Therefore, $k=2$ and $l=1$. Both structural sequences
\[ L_{<1} \stackrel{f_{<1}}{\longrightarrow} L \stackrel{f_{\geq 1}}{\longrightarrow} L_{\geq 1} \]
and
\[ L_{<2} \stackrel{f_{<2}}{\longrightarrow} L \stackrel{f_{\geq 2}}{\longrightarrow} L_{\geq 2} \]
are given by
\[ \{ \sigma_+ \} \hookrightarrow S^2 \stackrel{\id}{\longrightarrow} S^2. \]
The identity map is of course equivariant, but the inclusion of the suspension point is
equivariant as well, since this is a fixed point. It follows that the fiberwise (co)truncations
\[ \operatorname{ft}_{<1} E \longrightarrow E \longrightarrow \operatorname{ft}_{\geq 1} E \]
and
\[ \operatorname{ft}_{<2} E \longrightarrow E \longrightarrow \operatorname{ft}_{\geq 2} E \]
are both given by
\[ S^2_+ \stackrel{s_+}{\hookrightarrow} E \stackrel{\id}{\longrightarrow} E, \]
where $s_+$ is the section of $\pi:E\to S^2$ given by sending a point to the suspension point
$\sigma_+$ over it. Furthermore,
\[ Q_{\geq 1} E = Q_{\geq 2} E = E \cup_{S^2_+} D^3, \]
which is homotopy equivalent to complex projective space $\cp^2$. Indeed, a homotopy
equivalence is given by the quotient map
\[ Q_{\geq 1} E \stackrel{\simeq}{\longrightarrow}
  \frac{Q_{\geq 1} E}{D^3} \cong \frac{E}{S^2_+} \cong \frac{Dh_+ \cup_{S^3} Dh_-}{S^2_+}
  \cong D^4 \cup_{S^3} Dh_- = \cp^2. \]
The cohomology ring of $\cp^2$ is the truncated polynomial ring $\rat [x]/(x^3 =0)$
generated by 
\[ x\in H^2 (\cp^2) \cong \redH^2 (Q_{\geq 2} E) \cong \redH^{n-1-2} (Q_{\geq 1} E). \]
The square $x^2$ generates $H^4 (\cp^2)$, so by the injectivity of $C^*_{\geq 1} = C^*_{\geq 2},$
\[ C^*_{\geq 1} (x)\cup C^*_{\geq 2} (x) = C^*_{\geq 1} (x^2) \in H^4 (E) \]
is not zero. Thus the duality obstruction $\ob_2 (\pi,2,1)$ does not vanish.
It follows from Proposition \ref{prop.globalprodobvan} that $\pi:E\to S^2$ is in fact a nontrivial bundle, 
which can here of course also be seen directly.
Note that the Serre spectral sequence of any $S^2$-bundle over $S^2$ collapses at $E_2$.
Thus the obstructions $\ob_* (\pi,k,l)$ are able to detect twisting that is not detected by the
differentials of the Serre spectral sequence.
\end{example}

\section{Flat Bundles}

We have shown that the local duality obstructions vanish for product bundles.
We prove here that they also vanish for flat bundles, at least when the fundamental
group of the base is finite. The latter assumption can probably be relaxed, but we shall
not pursue this further here. A fiber bundle $\pi:E\to B$ with structure group $G$ is \emph{flat}
if its $G$-valued transition functions are locally constant.

\begin{theorem} \label{thm.obvanflatfinitefund}
Let $\pi:E \to B$ be a fiber bundle of  topological manifolds with structure group $G$, 
compact connected base $B$ and compact fiber $L$,
$\dim E=n-1,$ $b=\dim B,$ $c=n-b-1$.
For complementary perversities $\bar{p}, \bar{q},$ let
$k=c-\bar{p}(c+1),$ $l=c-\bar{q}(c+1).$ If
\begin{enumerate}
\item $L$ possesses $G$-equivariant Moore approximations of degree $k$ and of degree $l$,
\item $\pi$ is flat with respect to $G$, and
\item the fundamental group $\pi_1 (B)$ of the base is finite,
\end{enumerate}
then $\ob_i (\pi, k, l)=0$ for all $i$.
\end{theorem}
\begin{proof}
Let $\tB$ be the (compact) universal cover of $B$ and $\pi_1 = \pi_1 (B)$ the fundamental group.
By the $G$-flatness of $E$, there exists a monodromy representation
$\pi_1 \to G$ such that
\[ E = (\tB \times L)/\pi_1, \]
where $\tB \times L$ is equipped with the diagonal action of $\pi_1$, which is free.
If $M$ is any compact space on which a finite group $\pi_1$ acts freely,
then transfer arguments (using the finiteness of $\pi_1$) show that the
orbit projection $\rho: M\to M/\pi_1$ induces an isomorphism on rational cohomology,
\[ \rho^*: H^* (M/\pi_1) \stackrel{\cong}{\longrightarrow} H^* (M)^{\pi_1}, \]
where $H^* (M)^{\pi_1}$ denotes the $\pi_1$-invariant cohomology classes.
Applying this to $M=\tB \times L$, we get an isomorphism
\[ \rho^*: H^* (E) \stackrel{\cong}{\longrightarrow} H^* (\tB \times L)^{\pi_1}. \]
Using the monodromy representation, the $G$-cotruncation $L_{\geq k}$ becomes a
$\pi_1$-space with
\[ \cft{E}{k} = (\tB \times L_{\geq k})/\pi_1. \]
The closed subspace $\tB \times \star \subset \tB \times L_{\geq k},$ where $\star \in L_{\geq k}$ 
is the cone point, is $\pi_1$-invariant, since $\star$ is a fixed point of $L_{\geq k}$. 
Then a relative transfer argument applied to the
pair $(\tB \times L_{\geq k}, B\times \star)$ yields an isomorphism
\[ \rho^*: \redH^* (Q_{\geq k} E) = H^* (\operatorname{ft}_{\geq k} E, B)
   \stackrel{\cong}{\longrightarrow} H^* (\tB \times L_{\geq k}, \tB \times \star)^{\pi_1}. \]

Using the structural map $f_{\geq k}: L\to L_{\geq k},$ we define a map
\[ p_{\geq k} = \id \times f_{\geq k}: \tB \times L \longrightarrow \tB \times L_{\geq k}. \]
Since $f_{\geq k}$ is equivariant, the map $p_{\geq k}$ is $\pi_1$-equivariant with respect to
the diagonal action.
The diagram
\[ \xymatrix{
\tB \times L \ar[r]^\rho \ar[d]_{p_{\geq k}} & E \ar[d]^{c_{\geq k}} \\
\tB \times L_{\geq k} \ar[r]^\rho & \cft{E}{k}
} \]
commutes and induces on cohomology the commutative diagram
\begin{equation} \label{equ.eblft}
\xymatrix{
H^* (E) \ar[r]^{\rho^*}_{\cong} & H^* (\tB \times L)^{\pi_1} \\
H^* (\cft{E}{k}) \ar[r]^<<<<<{\rho^*}_<<<<<{\cong} \ar[u]^{c^*_{\geq k}} & 
H^* (\tB \times L_{\geq k})^{\pi_1} \ar[u]_{p^*_{\geq k}}
} \end{equation}
as we shall now verify: If $a\in H^* (\tB \times L_{\geq k})$ satisfies
$g^* (a)=a$ for all $g\in \pi_1,$ then the equivariance of $p_{\geq k}$ implies that
\[ g^* p^*_{\geq k} (a) = p^*_{\geq k} (g^* a) = p^*_{\geq k} (a), \]
which shows that indeed $p^*_{\geq k} (a) \in H^* (\tB \times L)^{\pi_1}$.
Similarly, there is a commutative diagram
\begin{equation} \label{equ.fteblq}
\xymatrix{
H^* (\cft{E}{k}) \ar[r]^{\rho^*}_{\cong} & 
     H^* (\tB \times L_{\geq k})^{\pi_1} \\
\redH^* (Q_{\geq k} E) \ar[r]^<<<<<{\rho^*}_<<<<<{\cong} \ar[u]^{\xi^*_{\geq k}} &
   H^* (\tB \times (L_{\geq k}, \star))^{\pi_1}. \ar[u]
} \end{equation}
Concatenating diagrams (\ref{equ.eblft}) and (\ref{equ.fteblq}), we obtain the commutative diagram
\[ \xymatrix{
H^* (E) \ar[r]^{\rho^*}_{\cong} & H^* (\tB \times L)^{\pi_1} \\
\redH^* (Q_{\geq k} E) \ar[r]^<<<<<{\rho^*}_<<<<<{\cong} \ar[u]^{C^*_{\geq k}} &
   H^* (\tB \times (L_{\geq k}, \star))^{\pi_1}. \ar[u]_{P^*_{\geq k}}
} \]
By the K\"unneth theorem, the cross product $\times$ is an isomorphism
\[ \times: H^* (\tB)\otimes H^* (L) \stackrel{\cong}{\longrightarrow} H^* (\tB \times L) \]
whose inverse is given by the Eilenberg-Zilber map $\EZ$.
Define a $\pi_1$-action on the tensor product $H^* (\tB)\otimes H^* (L)$ by
\[ g^* (a) := (\EZ \circ g^* \circ \times)(a),~ g\in \pi_1. \]
This makes the cross-product $\pi_1$-equivariant:
\[ \times \circ g^* (a) = \times \circ \EZ \circ g^* \circ \times (a) = g^* \circ \times (a). \]
Therefore, the cross-product restricts to a map
\begin{equation} \label{equ.restrcross}
\times: (H^* \tB \otimes H^* L)^{\pi_1} \longrightarrow H^* (\tB \times L)^{\pi_1}. 
\end{equation}
The Eilenberg-Zilber map is equivariant as well, since
\[ g^* \EZ (b) = \EZ \circ g^* \circ \times \circ \EZ (b) = \EZ \circ g^* (b). \]
Consequently, the Eilenberg-Zilber map restricts to a map
\begin{equation} \label{equ.restrez}
\EZ: H^* (\tB \times L)^{\pi_1} \longrightarrow (H^* \tB \otimes H^* L)^{\pi_1}. 
\end{equation}
Since $\times$ and $\EZ$ are inverse to each other, this shows in particular that the
restricted cross-product (\ref{equ.restrcross}) and the restricted Eilenberg-Zilber map (\ref{equ.restrez})
are isomorphisms.
All of these constructions apply just as well to $(L_{\geq k}, \star)$ instead of $L$.
By the naturality of the cross product, the square
\[ \xymatrix{
H^* (\tB \times L) & H^* \tB \otimes H^* L \ar[l]_\times^\cong \\
H^* (\tB \times (L_{\geq k}, \star)) \ar[u]^{P^*_{\geq k}} & H^* \tB \otimes H^* (L_{\geq k}, \star) 
  \ar[l]_\times^\cong \ar[u]_{\id \otimes f^*_{\geq k}} 
} \]
commutes. As we have seen, this diagram restricts to the various $\pi_1$-invariant subspaces.
In summary then, we have constructed a commutative diagram
\[ \xymatrix{
H^* (E) \ar[r]^{\rho^*}_{\cong} & H^* (\tB \times L)^{\pi_1} & 
     (H^* \tB \otimes H^* L)^{\pi_1} \ar[l]_\times^\cong \\
\redH^* (Q_{\geq k} E) \ar[r]^<<<<<{\rho^*}_<<<<<{\cong} \ar[u]^{C^*_{\geq k}} & 
   H^* (\tB \times (L_{\geq k}, \star))^{\pi_1} \ar[u]^{P^*_{\geq k}} & 
  (H^* \tB \otimes H^* (L_{\geq k}, \star))^{\pi_1} \ar[l]_\times^\cong \ar[u]_{\id \otimes f^*_{\geq k}}  
} \]
An analogous diagram is, of course, available for $Q_{\geq l} E$.

Let $x\in H^i (\tB \times (L_{\geq k}, \star))^{\pi_1},$ 
$y\in H^{n-1-i} (\tB \times (L_{\geq l}, \star))^{\pi_1}.$
Their images under the Eilenberg-Zilber map are of the form
\[ \EZ (x) = \sum_r b_r \otimes e^{\geq k}_r,~ b_r \in H^* (\tB),~ e^{\geq k}_r \in H^* (L_{\geq k},\star), \]
\[ \EZ (y) = \sum_s b'_s \otimes e^{\geq l}_s,~ b'_s \in H^* (\tB),~ e^{\geq l}_s \in H^* (L_{\geq l}, \star), \]
$\deg b_r + \deg e^{\geq k}_r = i,$
$\deg b'_s + \deg e^{\geq l}_s = n-1-i.$
Thus
\[ (\id \otimes f^*_{\geq k})\EZ (x)\cup (\id \otimes f^*_{\geq l})\EZ (y) =
\left( \sum_r b_r \otimes f^*_{\geq k} (e^{\geq k}_r) \right) \cup
   \left( \sum_s b'_s \otimes f^*_{\geq l} (e^{\geq l}_s) \right) \]
and
\begin{align*}
P^*_{\geq k} (x)\cup P^*_{\geq l} (y) &= 
  \times \circ (\id \otimes f^*_{\geq k})\EZ (x) \cup \times \circ (\id \otimes f^*_{\geq l})\EZ (y) \\
&= \left( \sum_r b_r \times f^*_{\geq k} (e^{\geq k}_r) \right) \cup
   \left( \sum_s b'_s \times f^*_{\geq l} (e^{\geq l}_s) \right) \\
&= \sum_{r,s} \pm (b_r \cup b'_s) \times (f^*_{\geq k} (e^{\geq k}_r) \cup f^*_{\geq l} (e^{\geq l}_s)).
\end{align*}
If $\deg f^*_{\geq k} (e^{\geq k}_r) + \deg f^*_{\geq l} (e^{\geq l}_s) <\dim L,$ then
$\deg b_r + \deg b'_s >\dim B$ and thus $b_r \cup b'_s =0$.
If $\deg f^*_{\geq k} (e^{\geq k}_r) + \deg f^*_{\geq l} (e^{\geq l}_s) >\dim L,$ then
trivially $f^*_{\geq k} (e^{\geq k}_r) \cup f^*_{\geq l} (e^{\geq l}_s)=0$.
Finally, if $\deg f^*_{\geq k} (e^{\geq k}_r) + \deg f^*_{\geq l} (e^{\geq l}_s) =\dim L,$
then $f^*_{\geq k} (e^{\geq k}_r) \cup f^*_{\geq l} (e^{\geq l}_s)=0$ by the defining properties
of cotruncation and the fact that $k$ and $l$ are complementary.
This shows that
\[ P^*_{\geq k} (x)\cup P^*_{\geq l} (y)=0. \]
For $\xi \in \redH^i (Q_{\geq k} E),$ $\eta \in \redH^{n-1-i} (Q_{\geq l} E),$
we find
\[ \rho^* (C^*_{\geq k}(\xi)\cup C^*_{\geq l} (\eta)) =
  \rho^* C^*_{\geq k}(\xi)\cup \rho^* C^*_{\geq l} (\eta) =
  P^*_{\geq k} (\rho^* \xi)\cup P^*_{\geq l} (\rho^* \eta)=0. \]
As $\rho^*$ is an isomorphism, 
\[ C^*_{\geq k}(\xi)\cup C^*_{\geq l} (\eta) =0. \]
\end{proof}

\section{Thom-Mather Stratified Spaces}
\label{sec.thommather}

In the present paper, intersection spaces will be constructed using the framework of Thom-Mather stratified spaces.
Such spaces are locally compact,
second countable Hausdorff spaces $X$ together with a Thom-Mather $C^\infty$-stratification, \cite{mathertopstab}.
We are concerned with \emph{two-strata pseudomanifolds}, which, in more detail, are understood to be
pairs $(X,\Sigma),$ where
$\Sigma\subset X$ is a closed subspace and a connected smooth manifold, and $X\setminus \Sigma$ is a smooth manifold
which is dense in $X$. The singular stratum $\Sigma$ must have codimension at least $2$ in $X$.
Furthermore, $\Sigma$ possesses control data consisting of an open neighborhood $T\subset X$ of $\Sigma$,
a continuous retraction $\pi:T\rightarrow \Sigma$, and a continuous distance function
$\rho:T\rightarrow \left[0,\infty \right)$ such that $\rho^{-1}\left(0\right)=\Sigma$.
The restriction of $\pi$ and $\rho$ to $T\setminus \Sigma$ are required to be smooth and
$(\pi,\rho): T\setminus \Sigma \to \Sigma \times (0,\infty)$ is required to be a submersion.
(Mather's axioms do \emph{not} require $(\pi,\rho)$ to be proper.)
Without appealing to the method of controlled vector fields required by
Thom and Mather for general stratified spaces, we shall prove directly that for two-strata spaces,
the bottom stratum $\Sigma$ possesses a locally trivial link bundle whose projection is induced by $\pi$. 

\begin{lemma}  \label{lem.preimsubm}
Let $f:M\to N$ be a smooth submersion between smooth manifolds and let
$Q\subset N$ be a smooth submanifold. Then $P=f^{-1}(Q)\subset M$ is a smooth submanifold
and $f|: P\to Q$ is a submersion.
\end{lemma}
\begin{proof}
A submersion is transverse to any submanifold. Thus $f$ is transverse to $Q$ and
$P=f^{-1}(Q)$ is a smooth submanifold of $M$.
The differential $f_*:T_x M\to T_{f(x)} N$ at any point $x\in P$ maps
$T_x P$ into $T_{f(x)} Q$ and thus induces a map 
$TM/TP \to TN/TQ$ of normal bundles. This map is a bundle isomorphism
(cf. \cite[Satz (5.12)]{broeckerjaenich}). An application of the four-lemma to the
commutative diagram with exact rows
\[ \xymatrix{
0 \ar[r] & T_x P \ar[d]_{f|_*} \ar[r] & T_x M \ar@{->>}[d]_{f_*} \ar[r] & T_x M/T_x P \ar[d]^{\cong} \ar[r] & 0 \\
0 \ar[r] & T_{f(x)} Q \ar[r] & T_{f(x)} N \ar[r] & T_{f(x)} N/T_{f(x)} Q \ar[r] & 0 \\
} \]
shows that $f|_*: T_x P\to T_{f(x)} Q$ is surjective for every $x\in P$.
\end{proof}

\begin{proposition} \label{prop.linkbundle}
Let $(X,\Sigma)$ be a Thom-Mather $C^\infty$-stratified pseudomanifold with two strata and control data
$(T,\pi,\rho)$. Then there exists a smooth function $\epsilon: \Sigma \to (0,\infty)$
such that the restriction $\pi: E\to \Sigma$ to
\[ E= \{ x\in T ~|~ \rho (x) = \epsilon (\pi (x)) \} \]
is a smooth locally trivial fiber bundle with structure
group $G=\operatorname{Diff}(L),$ the diffeomorphisms of $L=\pi^{-1}(s)\cap E,$ where $s\in \Sigma$.
\end{proposition}
\begin{proof}
If $\epsilon: \Sigma \to (0,\infty)$ is any function, we write
\[ T_\epsilon = \{ x\in T ~|~ \rho (x) < \epsilon (\pi (x)) \} \]
and
\[ \Sigma \times [0,\epsilon) = \{ (s,t)\in \Sigma \times [0,\infty) ~|~ 0\leq t < \epsilon (s) \}. \]
By \cite[Lemma 3.1.2(2)]{pflaum}, there exists a smooth $\epsilon$ such that
$(\pi,\rho): T_\epsilon \to \Sigma \times [0,\epsilon)$ is proper and surjective
(and still a submersion on $T_\epsilon \setminus \Sigma$ because $T_\epsilon \setminus \Sigma$ 
is open in $T \setminus \Sigma$).
(This involves only arguments of a point-set topological nature, but no controlled vector fields.
Pflaum's lemma provides only for a continuous $\epsilon$, but it is clear that on a smooth $\Sigma$,
one may take $\epsilon$ to be smooth.)
Setting
\[ E= \{ x\in T ~|~ \rho (x) = \smlhf \epsilon (\pi (x)) \} \subset T_\epsilon \setminus \Sigma, \]
we claim first that $\pi: E\to \Sigma$ is proper.
Let $\Gr \subset \Sigma \times [0,\infty)$ be the graph of $\smlhf \epsilon$. 
The continuity of $\epsilon$ implies that $\Gr$ is closed in $\Sigma \times [0,\infty)$ and
the smoothness of $\epsilon$ implies that $\Gr$ is a smooth submanifold.
From the description $E= (\pi, \rho)^{-1} (\Gr)$ we deduce that $E$ is closed in $T_\epsilon$.
The inclusion of a closed subspace is a proper map, and the composition of proper maps
is again proper. Hence the restriction of a proper map to a closed subspace is proper.
It follows that $(\pi,\rho):E\to \Sigma \times [0,\infty)$ is proper and then that
$(\pi,\rho):E\to \Gr$ is proper. The first factor projection $\pi_1 : \Sigma \times [0,\infty)\to \Sigma$
restricts to a diffeomorphism $\pi_1: \Gr \to \Sigma$, which is in particular a proper map.
The commutative diagram
\begin{equation} \label{dia.pirhopi1}
\xymatrix{
E \ar[rd]_{\pi} \ar[r]^{(\pi,\rho)} & \Gr \ar[d]_{\cong}^{\pi_1} \\
& \Sigma
} \end{equation}
shows that $\pi:E \to \Sigma$ is proper.

We prove next that $\pi:E\to \Sigma$ is surjective:
Given $s\in \Sigma$, the surjectivity of
$(\pi,\rho): T_\epsilon \to \Sigma \times [0,\epsilon)$ implies that there
is a point $x\in T_\epsilon$ such that $(\pi (x),\rho (x))=(s,\smlhf \epsilon (s))$, that is,
$\rho (x) =\smlhf \epsilon (\pi (x))$. This means that $x\in E$ and $\pi (x)=s$. 

By Lemma \ref{lem.preimsubm}, applied to the smooth map
$(\pi,\rho): T \setminus \Sigma \to \Sigma \times (0,\infty)$ and $Q=\Gr$,
$E= (\pi, \rho)^{-1} (\Gr)$ is a smooth submanifold and
$(\pi, \rho):E\to \Gr$ is a submersion. Using the diagram (\ref{dia.pirhopi1}),
$\pi: E\to \Sigma$ is a submersion. 

Applying Ehresmann's fibration theorem (for a modern exposition see \cite{dundas}) to the
proper, surjective, smooth submersion $\pi:E \to \Sigma$ yields the desired conclusion.
\end{proof}

We call the bundle given by Proposition \ref{prop.linkbundle} the \emph{link bundle} of $\Sigma$ in $X$.
The fiber is the \emph{link} of $\Sigma$. In this manner, $\Sigma$ becomes the base space
$B$ of a bundle and thus we will also use the notation $\Sigma =B$.
More generally, this construction evidently applies to the following class of spaces:

\begin{definition}\label{depth1}
A \emph{stratified pseudomanifold of depth $1$} is a tuple 
$\left(X,\Sigma_1,\cdots, \Sigma_r\right)$ such that the $\Sigma_i$ are mutually disjoint subspaces of $X$ such that $\left(X\setminus\left( \bigcup_{j\neq i}\Sigma_j\right), \Sigma_i\right)$ is a two strata pseudomanifold for every $i=1,\ldots,r$.
\end{definition}
In a depth $1$ space, every $\Sigma_i$ possesses its own link bundle.
\begin{definition}\label{Witt}
A stratified pseudomanifold of depth $1$, $\left(X,\Sigma_1,\cdots, \Sigma_r\right)$, is a \emph{Witt space} 
if the top stratum $X\setminus \bigcup \Sigma_i$ is oriented and the following condition is satisfied:
\begin{itemize}
\item For each $1\leq i \leq r$ such that $\Sigma_i$ has odd codimension $c_i$ in $X$, 
  the middle dimensional homology of the link $L_i$ vanishes:
	$$\rH{L_i}{\frac{c_i-1}{2}} = 0.$$
\end{itemize}
\end{definition}
Witt spaces were introduced by P. Siegel in \cite{SI}.
He assumed them to be endowed with a piecewise linear structure, as PL methods
allowed him to compute the bordism groups of Witt spaces.
We do not use these computations in the present paper.

\section{Intersection Spaces and Poincar\'e Duality}
\label{sec.intspacespd}

Let $\left(X,B\right)$ be a two strata pseudomanifold such that $X$ is $n$-dimensional, and $B$ is $b$-dimensional
and nonempty.
The Thom-Mather control data provide a tubular neighborhood $T$ of $B$ in $X$ and a distance function $\rho:T\rightarrow [0,\infty)$.  
Let $\epsilon:\Sigma =B \to (0,\infty)$ be the smooth function provided by
Proposition \ref{prop.linkbundle} such that
$\pi:E\to B$ is a fiber bundle, where
$E= \{ x\in T ~|~ \rho (x) = \epsilon (\pi (x)) \}$.
Let $M$ be the complement in $X$ of $T_\epsilon = \{ x\in T ~|~ \rho (x) < \epsilon (\pi (x)) \}$ and
let $L$ be the fiber of $\pi:E\to B$. By the surjectivity of $\pi$, $L$ is not empty.
The space $M$ is a smooth $n$-dimensional manifold with boundary $\partial M =E$.
Let $c=\dim L =n-1-b$.  Fix a perversity $\ov{p}$ satisfying the Goresky-MacPherson growth
conditions $\ov{p}(2)=0,$ $\ov{p}(s) \leq \ov{p}(s+1) \leq \ov{p}(s)+1$ for all $s\in \{ 2,3,\ldots \}$.
Set $k=c - \ov{p}\left(c+1\right)$. The growth conditions ensure that $k>0$.
Let $\ov{q}$ be the dual perversity to $\ov{p}$. The integer $l = c-\ov{q}\left(c+1\right)$ is positive.  
Assume that there exist $G$-equivariant Moore approximations of degree $k$ and $l$,  
$$\tr{f}{k}:\tr{L}{k}\rightarrow L \text{ and } \tr{f}{l}:\tr{L}{l}\rightarrow L$$
for some choice of structure group $G$ for the bundle $\pi:E\to B$.

We perform the fiberwise truncation and cotruncation
of Section \ref{fiberwise} on the link bundle $\pi: E=\partial M \rightarrow B$,  
use these constructions to define two incarnations of
intersection spaces, $I^{\ov{p}}X$ and $J^{\ov{p}}X$ associated to $X$, 
and show that they are homotopy equivalent.  The first, $I^{\ov{p}}X,$ agrees with the original definition 
given by the first author in \cite{BA} in all cases where they can be compared, the second $J^{\ov{p}}X$ has not 
been given before. It is introduced here to facilitate certain computations.

\begin{definition}\label{IX}
Define the map $\tr{\tau}{k}:\ft{E}{k} \rightarrow M$ to be the composition 
$$\xymatrix{
\tr{\tau}{k}: \ft{E}{k} \ar[r]^{\tr{F}{k} }& E=\partial M \ar@{^{(}->}[r]^(.6)i & M 
}\!\!,$$
where $i$ is the canonical inclusion of $\partial M$ as the boundary. 
Define $I^{\ov{p}}X$ to be the homotopy cofiber of $\tr{\tau}{k},$ i.e. the
homotopy pushout of the pair of maps
	$$\xymatrix{  \star & \ft{E}{k} \ar[l] \ar[r]^{\tr{\tau}{k}} & M.}$$
This is called the \emph{$\ov{p}$-intersection space} for $X$ defined via truncation.
If $E\cong B\times L$ is a product bundle, then this agrees with 
\cite[Definition 2.41]{BA}.
\end{definition}

\begin{definition}\label{JX}
In Section \ref{fiberwise}, we obtained the map $\ctr{C}{k}:E\rightarrow \ctr{Q}{k}E$.  
Define the $\ov{p}$-intersection space for $X$ \emph{via cotruncation}, $J^{\ov{p}}X$, to be the space obtained 
as the homotopy pushout of
	$$\xymatrix{ \ctr{Q}{k} & E \ar[l]_{\ctr{C}{k}} \ar@{^{(}->}[r]^{i} & M.}$$
\end{definition}

We have the following diagram of topological spaces, commutative up to homotopy, in which every square is a homotopy pushout square:
	$$\xymatrix{ \ft{E}{k} \ar[r]^{\tr{F}{k}} \ar[d]^{\tr{\pi}{k}} & 
E \ar[d]^{\ctr{c}{k}} \ar[r]^i & M \ar[dd]^{\ctr{\eta}{k}} \\ 
B \ar[r]^{\sigma} \ar[d] & \cft{E}{k} \ar[d]^{\ctr{\xi}{k}} & \\ 
\star \ar[r]^{[c]} & \ctr{Q}{k}E \ar[r]^{\ctr{\nu}{k}} & J^{\ov{p}}X, }$$
where $\ctr{\eta}{k}$ and $\ctr{\nu}{k}$ are defined to be the maps coming from the 
definition of $J^{\ov{p}}X$ as a homotopy pushout.  

\begin{lemma}  \label{lem.ipxhejpx}
The canonical collapse map $J^{\ov{p}}X \to I^{\ov{p}}X$ is a homotopy equivalence.
\end{lemma}
\begin{proof}
By construction, the space $J^{\ov{p}}X$ contains the cone on $B$, $cB,$ as a subspace and 
$(J^{\ov{p}}X, cB)$ is an NDR-pair. Since $cB$ is contractible, the collapse map
$J^{\ov{p}}X \to J^{\ov{p}}X/cB$ is a homotopy equivalence. The quotient
$J^{\ov{p}}X/cB$ is homeomorphic to $I^{\ov{p}}X$.
\end{proof}

The sequence 
\[ \operatorname{ft}_{<k} E \stackrel{\tau_{<k}}{\longrightarrow} M
  \longrightarrow \operatorname{cone} (\tau_{<k}) = I^{\ov{p}} X \]
induces a long exact sequence
\begin{equation} \label{JPXcmv}
\xymatrix{ \ar[r] & 
H^{r-1} (\ft{E}{k}) \ar[r]^{\delta^{\ov{p},r}} & \rrC{I^{\ov{p}}X}{r}  \ar[r]^{\ctr{\eta}{k}^r}& 
H^r (M)  \ar[r]^{\tr{\tau}{k}^r} & H^r (\ft{E}{k}) \ar[r]&. }
\end{equation}
Furthermore, we can define $\hat{M}$ to be the homotopy pushout of the pair of maps 
	$$\xymatrix{ \star & \partial M = E \ar[l] \ar[r]^i & M.}$$  
This is nothing but the space $M$ with a 
cone attached to the boundary.  Define $J^{-1}X$ to be the homotopy pushout obtained from the pair of maps 
	$$\xymatrix{ \star & \ctr{Q}{k}E \ar[l] \ar[r]^{\ctr{\nu}{k}} & J^{\ov{p}}X. }$$
\begin{lemma}  
The canonical collapse map $J^{-1}X \to \hat{M}$ is a homotopy equivalence.
\end{lemma}
\begin{proof}
The space $J^{-1}X$ contains the cone $cQ_{\geq k} E$ as a subspace and 
$(J^{-1} X, cQ_{\geq k} E)$ is an NDR-pair. Thus the collapse map
$J^{-1}X \to J^{-1}X/cQ_{\geq k} E$ is a homotopy equivalence. The quotient
$J^{-1}X/cQ_{\geq k} E$ is homeomorphic to $\hat{M}$.
\end{proof}
By the lemma, using $l$ and $\ov{q}$ instead of $k$ and $\ov{p}$, we have the 
long exact sequence \eqref{pushoutMVred} associated to $J^{-1}X$:
\begin{equation}  \label{JQXhmv} 
\xymatrix{ \ar[r] & \rrH{\ctr{Q}{l}E}{r}  \ar[r]^{\ctr{\nu}{l,r}}   & 
\rrH{J^{\ov{q}}X}{r}   \ar[r]^{\ctr{\zeta}{l,r}} & 
\rH{M,\partial M}{r}\ar[r]^{\delta^{\ov{q}}_r} & 
\rrH{\ctr{Q}{l}E}{r-1} \ar[r] &, }
\end{equation}
where $\ctr{\zeta}{l}$ is the composition of the map $J^{\ov{q}}X\rightarrow J^{-1}X,$ 
defined by $J^{-1}X$ as a homotopy pushout, with the collapse map
$J^{-1} X \stackrel{\simeq}{\longrightarrow} \hat{M}$.
In the sequence, we have identified $\redH_r (\hat{M}) \cong \rH{M,\partial M}{r}$.  

\begin{theorem}\label{JXPD}
Let $\left(X,B\right)$ be a compact, oriented, two strata pseudomanifold of dimension $n$. 
Let $\ov{p}$ and $\ov{q}$ be 
complementary perversities, and $k=c-\ov{p}\left(c+1\right)$, $l=c-\ov{q}\left(c+1\right)$,
where $c=n-1-\dim B$.
Assume that equivariant Moore approximations to $L$ of degree $k$ and degree $l$ exist. 
If the local duality obstructions $\ob_* (\pi,k,l)$ of the link bundle $\pi$ vanish, 
then there is a global Poincar\'e duality isomorphism 
\begin{equation} \label{equ.ipiqpd}
\rrC{I^{\ov{p}}X}{r} \cong \rrH{I^{\ov{q}}X}{n-r}.
\end{equation}
\end{theorem}
\begin{proof}
We achieve this by pairing the sequence (\ref{JPXcmv}) with the sequence (\ref{JQXhmv}) 
(observing Lemma \ref{lem.ipxhejpx}) and using the five lemma.  
Consider the following diagram of solid arrows whose rows are exact:
\begin{equation} \label{dia.globalpd}
\xymatrix@C=20pt{ 
\ar[r] & H^{r-1} (\ft{E}{k}) \ar[r]^{\delta^{\ov{p},*}}\ar[d]^{D^{r-1}_{k,l}}_\cong 
   & \rrC{I^{\ov{p}}X}{r}  \ar@{-->}[d]^{D^r_{IX}}\ar[r]^{\ctr{\eta}{k}^*}& 
H^r (M)  \ar[r]^{\tr{\tau}{k}^*} \ar[d]^{D_{M}^r}_\cong & 
H^r (\ft{E}{k}) \ar[d]^{D^{r}_{k,l}}_\cong & \\
\ar[r] & \rrH{\ctr{Q}{l}E}{n-r}  \ar[r]^{\ctr{\nu}{l,*}}   & \rrH{I^{\ov{q}}X}{n-r}   \ar[r]^{\ctr{\zeta}{l,*}}& 
\rH{M,\partial M}{n-r}\ar[r]^{\delta^{\ov{q}}_{*}}& \rrH{\ctr{Q}{l}E}{n-r-1}  & 
} \end{equation}
Here $D^r_{k,l}$ comes from Proposition \ref{bundle duality}, and $D_M^r$ comes from the classical Lefschetz 
duality for manifolds with boundary. The solid arrow square on the right can be written as
$$\xymatrix{ 
H^r (M) \ar[r]^{i^*} \ar[d]^{D^r_M}_\cong & 
H^r (\partial M) \ar[r]^{\tr{F}{k}^*} \ar[d]^{D^r_{\partial M}}_\cong & 
H^r (\ft{E}{k}) \ar[d]^{D^r_{k,l}}_\cong \\ 
\rH{M,\partial M}{n-r} \ar[r]^{\delta^{M,\partial M}_{*}} & 
H_{n-r-1} (\partial M) \ar[r]^{\ctr{C}{l,*}} & \rrH{\ctr{Q}{l}E}{n-r-1}	
}$$
The left square commutes by classical Poincar\'e-Lefschetz duality, and the right square 
commutes by 
Proposition \ref{prop.dob0} and Proposition \ref{prop.obzerodisdkl}, since $\ob_* (\pi,k,l)=0$.
Thus diagram (\ref{dia.globalpd}) commutes. By e.g. \cite[Lemma 2.46]{BA}, 
we may find a map $D^r_{IX}$ to fill in the dotted 
arrow so that the diagram commutes.  By the five lemma, $D^r_{IX}$ is an isomorphism.
\end{proof}

It does not follow from this proof that for a $4d$-dimensional Witt space $X$ the associated intersection form
$\redH_{2d} (IX)\times \redH_{2d} (IX)\to \rat$ is symmetric, where $IX = I^{\bar{m}} X =
I^{\bar{n}} X$. In Section \ref{sec.signintsp}, however, we shall prove that the isomorphism
(\ref{equ.ipiqpd}) can always be constructed so as to yield a symmetric intersection form
(cf. Proposition \ref{prop.dixsymmetric}).

\section{Moore Approximations and the Intersection Homology Signature}
\label{sec.ihsign}

Assume that $\left(X,B\right)$ is a two-strata Witt space with $\dim X = n=4d$, $d> 0$, and 
$\dim B = b$, then $c=4d-1-b=\dim L$.  If we use the upper-middle perversity $\ov{n}$ and the 
lower-middle perversity $\ov{m}$, which are complementary, we get the associated pair of integers 
$k = \lfloor \frac{c+1}{2} \rfloor$ and $l =\lceil \frac{c+1}{2} \rceil$.  When $c$ is odd then 
$k=l=\frac{c+1}{2}$, and when $c$ is even then $k=c/2$ and $l=k+1$.  Notice that the codimension of 
$B$ in $X$ is $c+1$.  So the Witt condition says that when $c$ is even then $\rH{L}{\frac{c}{2}}=0$.  
In this case if an equivariant Moore approximation of degree $k$ exists, then so does one of 
degree $k+1 =l$ and they can be chosen to be equal.  Therefore, when $X$ satisfies the Witt condition 
and an equivariant Moore approximation to $L$ of degree $k$ exists, we can construct 
$I^{\ov{m}}X = I^{\ov{n}}X$ and
$J^{\ov{m}}X = J^{\ov{n}}X$.  We denote the former space $IX$ and the latter $JX$ and call 
this homotopy type the \emph{intersection space} associated to the Witt space $X$.

The cone bundle $DE$ is nothing but $\cft{E}{c+1}$ with $\tr{L}{c+1}=L$.  
Note that when $E=\partial M$ as above, then $DE$ is a two strata space with boundary $\partial DE = \partial M$, 
and we can realize $X$ as the pushout of the pair of maps $\xymatrix{ M & \partial M \ar[l]_i \ar[r]^{\ctr{c}{c+1}} & DE}$.  
Thus $\partial M$ is bi-collared in $X$ and by Novikov additivity, Prop. II,3.1 \cite{SI}, we have that the 
intersection homology Witt element $w_{IH}$, 
defined in I,4.1 \cite{SI}, is additive over these parts,
\begin{equation} \label{equ.novadd}
w_{IH}\left(X\right) = w_{IH}(\hat{M}) + w_{IH}\left(TE\right)\in W\left(\Q\right),
\end{equation}
where the Thom space $TE$ is $DE$ with a cone attached to its boundary, 
and $W\left(\Q\right)$ is the Witt group of $\Q$. 
When $X$ is Witt, we write $IH_* (X)$ for $I^{\bar{m}} H_* (X) = I^{\bar{n}} H_* (X)$.

\begin{proposition}\label{TEhomology1}
If an equivariant Moore approximation to $L$ of degree $k=\lfloor \frac{1}{2} (\dim L +1) \rfloor$ 
exists, then the middle degree, 
middle perversity intersection homology of the $n=4d$-dimensional Witt space $TE$ vanishes,
$$\ImH{TE}{2d} = 0.$$
\end{proposition}
\begin{proof}
In this proof we use the notation $\dot{c}E$ and $\dot{D}E$ to mean the open cone on $E$ and the 
open cone bundle associated to $E$.  According to (\ref{openclosed}), 
$I^{\ov{p}}H_r (\dot{D}E) \cong \IH{DE}{r}$, and 
$\IH{\dot{c}E}{r}\cong \IH{cE}{r}$ for all $r\geq 0$.  
Hence, as in the proof of Proposition \ref{Cone Formula}, we can 
identify the long exact sequence of intersection homology groups associated to the pair 
$(\dot{D}E,\dot{D}E\setminus B)$ with the same sequence associated to the 
$\partial$-stratified pseudomanifold $\left(DE,E\right)$ from \eqref{boundary PD}.

Define open subsets $U,V$ of $TE$ by 
$U=TE\setminus B=\dot{c}E$ and $V=TE\setminus c=\dot{D}E$, where $c$ is the cone point. 
Then $TE=U\cup V$ and $U\cap V = E\times (-1,1)$.  The Mayer-Vietoris sequence associated to the pair $\left(U,V\right)$ gives
\begin{equation} \label{TEMV}
\xymatrix{ 
\ar[r] & \rH{E}{r} \ar[r]^(.4){i_r^{TE}} & IH_r(\dot{D}E) \oplus\ImH{\dot{c}E}{r}\ar[r]^(.65){j_r^{TE}} 
& \ImH{TE}{r} \ar[r]^{\delta_r^{TE}} & \rH{E}{r-1} \ar[r] &
}\end{equation}
Here we have identified $\ImH{E\times(-1,1)}{r}\cong \rH{E}{r}$.  
After making the identifications as decribed in the previous paragraph, 
the map $i_r^{TE}= i_r^{DE}\oplus i_r^{cE}$ is identified as the sum of the maps coming 
from the sequences associated to the pairs $\left(DE,E\right)$ and $\left(cE,E\right)$ respectively.  
In degrees $r<2d$ we know from Proposition \ref{Cone Formula} that $i_r^{cE}$ is 
an isomorphism $\rH{E}{r}  =\ImH{cE}{r}$.  Thus $i_r^{TE}$ is injective for $r<2d$.  
Consequently, when $r=2d,$ we have an exact sequence
\[
\xymatrix{ 
\cdots \ar[r] & \rH{E}{2d} \ar[r] & 
\ImH{DE}{2d} \oplus IH_{2d} (cE)   \ar[r]
& \ImH{TE}{2d} \ar[r] & 0.
} \]
By the cone formula for intersection homology, $IH_{2d} (cE)=0$, since 
$2d = \dim E - \ov{m}(\dim E +1)$. Now by Proposition \ref{relationtoIH},
the map $H_{2d} (E)\to IH_{2d} (DE)$ is surjective.
\end{proof}

\begin{corollary}\label{IHsignature}
Let $X$ be a compact, oriented, $n=4d$-dimensional stratified pseudomanifold of depth $1$ 
which satisfies the Witt condition.  If equivariant Moore approximations of degree 
$k=\lfloor \frac{1}{2} (\dim L +1) \rfloor$ to the links of the singular set exist, then
	$$w_{IH}\left(X\right)=w_{IH} (\hat{M}) \in W\left(\Q\right).$$
In particular, the signature of the intersection form on intersection homology satisfies
	$$\sigma_{IH}\left(X\right)= \sigma_{IH} (\hat{M}).$$
\end{corollary}

\begin{proof}
If $\ImH{TE}{2n}=0$, then $w_{IH}\left(TE\right)=0$. The assertion follows from
Novikov additivity (\ref{equ.novadd}).
\end{proof}

\begin{example}
Let $X=\mathbb{CP}^2$ be complex projective space with $B=\mathbb{CP}^1 \subset X$
as the bottom stratum, so that the link bundle is the Hopf bundle over $B$. Then
\[ \sigma_{IH} (X) = \sigma (\mathbb{CP}^2)=1, \]
but
\[ \sigma (M,\partial M) = \sigma (D^4, S^3) =0. \]
Indeed, the link $S^1$ in the Hopf bundle has no middle-perversity equivariant Moore-approximation
because the Hopf bundle has no section.
\end{example}

\section{The Signature of Intersection Spaces}
\label{sec.signintsp}

Theorem 2.28 in \cite{BA} states that for a closed, oriented, $4d$-dimensional Witt space $X$ with only isolated
singularities, the signature of the symmetric nondegenerate intersection form
$\redH_{2d} (IX)\times \redH_{2d} (IX)\to \rat$ equals the signature of the Goresky-MacPherson-Siegel
intersection form $IH_{2d} (X)\times IH_{2d} (X)\to \rat$ on middle-perversity intersection
homology. In fact, both are equal to the Novikov signature of the top stratum. We shall here
generalize that theorem to spaces with twisted link bundles that allow for equivariant Moore approximation.

\begin{definition}
Define the signature of a $4d$-dimensional manifold-with-boundary $\left(M,\partial M\right)$ to be 
	$$\sigma\left(M,\partial M\right) = \sigma\left( \beta \right),$$
where $\beta$ is the bilinear form
$$\beta: \im j_* \times \im j_* \rightarrow \Q,~ 
   (j_*  v, j_* w) \mapsto (d_M (v)) (j_* w),$$
the homomorphism
\[ j_*: H_{2d} (M) \longrightarrow H_{2d} (M,\partial M) \]
is induced by the inclusion, and
\[ d_M: H_{2d} (M) \longrightarrow H^{2d} (M, \partial M) \]
is Lefschetz duality.
This is frequently referred to as the \emph{Novikov signature} of $(M,\partial M)$.
It is well-known (\cite{SI}) that $\sigma (M, \partial M) = \sigma_{IH} (\hat{M})$.
\end{definition}

Let $(X,B)$ be a two strata Witt space with $\dim X = n =4d$, $\dim B=b$. We assume that an 
equivariant Moore approximation of degree $k = 4d-b-1-\ov{m}\left(4d-b\right)$ exists for the link $L$ of $B$ in $X$,
and that the local duality obstruction $\ob_* (\pi, k,k)$ vanishes.
As discussed in the previous section, this implies that the intersection space $IX$ exists and is well defined.  
Theorem \ref{JXPD} asserts that $IX$ satisfies Poincar\'e duality
\[ d_{IX}: \redH_{2d} (IX) \stackrel{\cong}{\longrightarrow} \redH^{2d} (IX). \]
We shall show (Proposition \ref{prop.dixsymmetric}) that $d_{IX}$ can in fact be so constructed 
that the associated intersection form
on the middle-dimensional homology is symmetric. One may then consider its signature:

\begin{definition}
The signature of the space $IX,$
	$$\sigma\left(IX\right) = \sigma\left( \beta \right),$$	
	is defined to be the signature of the symmetric bilinear form
	$$\beta: \redH_m (IX) \times \redH_m (IX) \rightarrow \Q,$$
with $m=2d,$ defined by 
	$$\beta (v,w)=d_{IX} (v)(w)$$
for any $v,w \in \redH_m (IX)$.  Here we have identified 
$\redH^m (IX) \cong \redH_m (IX)^\dagger$ via the universal coefficient theorem.
\end{definition}

\begin{theorem}\label{HIsignature}
The signature of $IX$ is supported away from the singular set $B$, that is,
	$$\sigma\left(IX\right) = \sigma\left(M,\partial M\right).$$
\end{theorem}
Before we prove this theorem, we note that in view of Corollary \ref{IHsignature}, we
immediately obtain:
\begin{corollary}
If a two-strata Witt space $(X,B)$ allows for middle-perversity equivariant Moore-approximation of
its link and has vanishing local duality obstruction, then
\[ \sigma_{IH}(X) = \sigma (IX). \]
\end{corollary}

The rest of this section is devoted to the proof of Theorem \ref{HIsignature}.
We build on the method of Spiegel \cite{SP}, which in turn is partially based
on the methods introduced in the proof of \cite[Theorem 2.28]{BA}.
Regarding notation, we caution that the letters $i$ and $j$ will both denote
certain inclusion maps and appear as indices. This cannot possibly lead to any confusion.

Let $\{ e_1, \ldots, e_r \}$ be any basis for $j_* H_m (M)$, where
\[ j_*: H_{m} (M) \longrightarrow H_{m} (M,\partial M) \]
is induced by the inclusion.
For every $i=1,\ldots, r,$ pick a lift $\overline{e}_i \in H_m (M),$ 
$j_* (\overline{e}_i) = e_i$.
Then $\{ \overline{e}_1,\ldots, \overline{e}_r \}$ is a linearly independent set in $H_m (M)$ and
\begin{equation} \label{equ.oekerjiszero}
\rat \langle \overline{e}_1,\ldots, \overline{e}_r \rangle \cap \ker j_* = \{ 0 \}. 
\end{equation}
Let 
\[ d_M: H_m (M) \stackrel{\cong}{\longrightarrow} H^m (M,\partial M)=H_m (M,\partial M)^\dagger \]
be the Lefschetz duality isomorphism, i.e. the inverse of 
\[ D'_M: H^m (M,\partial M) \stackrel{\cong}{\longrightarrow} H_m (M), \]
given by capping with the fundamental 
class $[M,\partial M]\in H_{2m} (M,\partial M)$.
Let 
\[ d'_M: H_m (M,\partial M) \stackrel{\cong}{\longrightarrow} H^m (M) \]
be the inverse of 
\[ D_M: H^m (M) \stackrel{\cong}{\longrightarrow} H_m (M,\partial M), \]
given by capping with the fundamental class.
We shall make frequent use of the symmetry identity 
\[ d_M (v)(w) = d'_M (w)(v), \]
$v\in H_m (M),$ $w\in H_m (M,\partial M),$ which holds since the cup product
of $m$-dimensional cohomology classes commutes as $m=2d$ is even.
The commutative diagram
\[ \xymatrix{
H_m (M) \ar[r]^{j_*} \ar[d]_{d_M} & H_m (M,\partial M) \ar[d]^{d'_M} \\
H^m (M,\partial M) \ar[r]_{j^*} & H^m (M)
} \]
implies that the symmetry equation
\[ d_M (\overline{e}_i)(e_j) = d_M (\overline{e}_j)(e_i) \]
holds, as the calculation
\begin{align*} 
d_M (\overline{e}_i)(e_j) &= d_M (\overline{e}_i)( j_* \overline{e}_j)
 = j^* d_M (\overline{e}_i)(\overline{e}_j) = d'_M (j_* \overline{e}_i)(\overline{e}_j) \\
 &= d'_M (e_i)(\overline{e}_j) = d_M (\overline{e_j})(e_i) 
\end{align*}
shows.

In the proof of \cite[Theorem 2.28]{BA}, the first author introduced the annihilation subspace
$Q\subset H_m (M,\partial M)$,
\[ Q = \{ q\in H_m (M,\partial M) ~|~ d_M (\overline{e}_i)(q)=0 \text{ for all } i \}. \]
It is shown on p. 138 of \emph{loc. cit.} that one obtains an internal direct sum decomposition
\[ H_m (M,\partial M) = \im j_* \oplus Q. \]
Let $L\subset \redH_m (IX)$ be the kernel of the map
\[ \zeta_{\geq k*}: \redH_m (IX) \longrightarrow H_m (M,\partial M). \]
Once we have completed the construction of a symmetric intersection form,
$L$ will eventually be shown to be a Lagrangian subspace of an appropriate subspace of
$\redH_m (IX)$. Let $\{ u_1, \ldots u_l \}$ be any basis for $L$.

We consider the commutative diagram
\begin{equation} \label{equ.ijetazetanuctauf}
\xymatrix{
& H_m (M,\partial M) \ar@{=}[r] & H_m (M,\partial M) & \\
H_m (\ftr_{<k} E) \ar[r]^{\tau_{<k*}} & H_m (M) \ar[u]^{j_*} \ar[r]^{\eta_{\geq k*}} &
   \redH_m (IX) \ar[u]^{\zeta_{\geq k*}}  \ar[r]^{\delta_*}  &
   H_{m-1} (\ftr_{<k} E) \\
H_m (\ftr_{<k} E) \ar@{=}[u] \ar@{^{(}->}[r]^{F_{<k*}} &
   H_m (\partial M) \ar[u]^{i_*} \ar@{->>}[r]^{C_{\geq k*}} &
  \redH_m (Q_{\geq k} E) \ar[u]^{\nu_{\geq k*}} \ar[r]^{\delta_* =0} &
  H_{m-1} (\ftr_{<k} E) \ar@{=}[u] 
} \end{equation}
The rows and columns are exact and we have used Lemma \ref{lem.flkcgkexact}.
By exactness of the right hand column, the basis elements $u_j$ can be lifted to
$\redH_m (Q_{\geq k} E)$, and by the surjectivity of $C_{\geq k*},$ these lifts
can be further lifted to $H_m (\partial M)$. In this way, we obtain linearly independent
elements $\overline{u}_1,\ldots, \overline{u}_l$ in $H_m (\partial M)$ such that
\[ \eta_{\geq k*} i_* (\overline{u}_j) = \nu_{\geq k*} C_{\geq k*} (\overline{u}_j) = u_j \]
for all $j$.
Setting
\[ w^j = d_M (i_* (\overline{u}_j)) \]
yields a linearly independent set $\{ w^1,\ldots, w^l \} \subset H^m (M,\partial M).$
From now on, let us briefly write $\eta_*, \zeta_*,$ etc., for $\eta_{\geq k*}, \zeta_{\geq k*},$ etc.
Since $\eta_* i_* (\overline{u}_j)=u_j,$ we have
\[ \rat \langle i_* (\overline{u}_1),\ldots, i_* (\overline{u}_l) \rangle \cap \ker \eta_* = \{ 0 \}. \]
Together with (\ref{equ.oekerjiszero}), and noting $\ker \eta_* \subset \ker j_*,$
this shows that there exists a linear subspace
$A\subset H_m (M)$ yielding an internal direct sum
decomposition
\begin{equation} \label{equ.hommdecomp}
H_m (M) = \rat \langle i_* (\overline{u}_1),\ldots, i_* (\overline{u}_l) \rangle \oplus
  \ker \eta_* \oplus \rat \langle \overline{e}_1, \ldots, \overline{e}_r \rangle \oplus A. 
\end{equation}
Setting
\[ Z =  \ker \eta_* \oplus \rat \langle \overline{e}_1, \ldots, \overline{e}_r \rangle \oplus A, \]
we have
\[ H_m (M) = \rat \langle i_* (\overline{u}_1),\ldots, i_* (\overline{u}_l) \rangle \oplus Z, \]
such that 
\begin{equation} \label{equ.keretaoveinz}
\ker \eta_* \subset Z \text{ and } \rat \langle \overline{e}_1, \ldots, \overline{e}_r \rangle \subset Z. 
\end{equation}
Choose a basis $\{ \widetilde{z}_1, \ldots, \widetilde{z}_s \}$ of $Z$ and put
$z^j = d_M (\widetilde{z}_j) \in H^m (M,\partial M)$.
Then $\{ z^1,\ldots z^s \}$ is a basis for $d_M (Z)$ and
\[ H^m (M,\partial M) = \rat \langle w^1, \ldots, w^l \rangle \oplus 
     \rat \langle z^1, \ldots, z^s \rangle. \]
 Let 
\[ \{ w_1,\ldots, w_l, z_1,\ldots, z_s \} \subset H_m (M,\partial M) \]
be the dual basis of $\{ w^1, \ldots, w^l, z^1, \ldots, z^s \}$, that is,
\begin{equation} \label{equ.dualwz}
w^i (w_j)=\delta_{ij},~ z^i (z_j)=\delta_{ij},~ w^i (z_j)=0,~ z^i (w_j)=0. 
\end{equation}

\begin{lemma} \label{lem.winimzeta}
The set $\{ w_1,\ldots, w_l \}$ is contained in the image of $\zeta_*$.
\end{lemma}
\begin{proof}
In view of the commutative diagram
\[ \xymatrix{
\redH_m (IX) \ar[r]^{\zeta_*} & H_m (M,\partial M) \ar[r]^{\delta_*} \ar[d]_{d'_M}^\cong &
  \redH_{m-1} (Q_{\geq k} E) \\
& H^m (M) \ar[r]^{\tau^*} & H^m (\ftr_{<k} E), \ar[u]_{D_{k,k}}^\cong
} \]
it suffices to show that $\delta_* (w_j)=0$, since the top row is exact.
Let $x\in H_m (\ftr_{<k} E)$ be any element. Then $\tau_* x \in \ker \eta_* \subset Z$,
so $d_M (\tau_* x)(w_j)=0$ by (\ref{equ.dualwz}). Consequently,
\[ (\tau^* d'_M (w_j))(x) = d'_M (w_j)(\tau_* x)=d_M (\tau_* x)(w_j)=0. \]
It follows that $\tau^* d'_M (w_j)=0$ and in particular
\[ \delta_* (w_j) = D_{k,k} \tau^* d'_M (w_j)=0. \]
\end{proof}
Suppose that $v\in \ker \zeta_* \cap \eta_* \langle \overline{e}_1,\ldots, \overline{e}_r \rangle$.
Then $v$ is a linear combination $v=\eta_* \sum \lambda_i \overline{e}_i$ and
\[ 0= \zeta_* (v) = \zeta_* \eta_* \sum \lambda_i \overline{e}_i
 = \sum \lambda_i j_* (\overline{e}_i) = \sum \lambda_i e_i. \]
Thus $\lambda_i =0$ for all $i$ by the linear independence of the $e_i.$ This shows that 
\[ L \cap \eta_* \langle \overline{e}_1,\ldots, \overline{e}_r \rangle = \{ 0 \}.  \]
Therefore, it is possible to choose a direct sum complement $W\subset \redH_m (IX)$ of $L=\ker \zeta_*,$
\begin{equation} \label{equ.decomphixlw} 
\redH_m (IX) = L\oplus W, 
\end{equation}
such that
\begin{equation} \label{equ.etaoveiinw}
 \eta_* \langle \overline{e}_1,\ldots, \overline{e}_r \rangle \subset W.
\end{equation}
The restriction
\[ \zeta_*|_W: W \longrightarrow \im \zeta_* \]
is then an isomorphism and thus by Lemma \ref{lem.winimzeta}, we may define
\[ \overline{w}_j = (\zeta_*|_W)^{-1} (w_j). \]
We define subspaces $V, L' \subset W$ by
\[ V=  (\zeta_*|_W)^{-1} (\im j_*),~   L' = (\zeta_*|_W)^{-1} (Q\cap \im \zeta_*). \]
Recall that $\{ e_1, \ldots, e_r \}$ is a basis of $\im j_*$. Setting
\[ v_j = (\zeta_*|_W)^{-1} (e_j), \]
yields a basis $\{ v_1, \ldots, v_r \}$ for $V$.
From
\[ \zeta_* (v_i) = e_i = j_* (\overline{e}_i) = \zeta_* \eta_* (\overline{e}_i) \]
it follows that
\[ v_i = \eta_* (\overline{e}_i), \]
since both $v_i$ and $\eta_* (\overline{e}_i)$ are in $W$ and $\zeta_*$ is injective on $W$.
The decomposition $H_m (M,\partial M)=\im j_* \oplus Q$ induces a decomposition
\[ \im \zeta_* = (\im j_* \oplus Q)\cap \im \zeta_* =
   \im j_* \oplus (Q\cap \im \zeta_*). \]
Applying the isomorphism $(\zeta_*|_W)^{-1}$, we receive a decomposition
\[ W = (\zeta_*|_W)^{-1}(\im j_*) \oplus (\zeta_*|_W)^{-1}(Q\cap \im \zeta_*)
       =  V \oplus L'. \]
By (\ref{equ.decomphixlw}), we arrive at a decomposition
\[ \redH_m (IX) = L \oplus V \oplus L'. \]
\begin{lemma}
The set $\{ \overline{w}_1,\ldots, \overline{w}_l \} \subset W$ is contained in $L'$.
\end{lemma}
\begin{proof}
By construction of $L'$, we have to show that $\zeta_* (\overline{w}_j)\in Q$ for all $j$.
Now $\zeta_* (\overline{w}_j)=w_j$, so by construction of $Q$, we need to demonstrate
that $d_M (\overline{e}_i)(w_j)=0$ for all $i$. By (\ref{equ.keretaoveinz}), 
$d_M (\overline{e}_i) \in d_M (Z),$ whence the result follows from (\ref{equ.dualwz}).
\end{proof}

\begin{lemma}
The set $\{ \overline{w}_1,\ldots, \overline{w}_l \} \subset W$ is a basis for $L'$.
\end{lemma}
\begin{proof}
The preimages
$\overline{w}_j = (\zeta_*|_W)^{-1} (w_j)$ under the isomorphism $\zeta_*|_W$ are linearly
independent since $\{ w_1, \ldots, w_l \}$ is a linearly independent set. In particular,
$\dim L' \geq l$. It remains to be shown that $\dim L' \leq l$.
Standard linear algebra provides the inequality
\[ \rk \eta_* \leq \dim \ker \zeta_* + \rk (\zeta_* \eta_*), \]
valid for the composition of any two linear maps.
As $\zeta_* \eta_* = j_*,$ we may rewrite this as
\begin{equation} \label{equ.rketaleqlrkj}
 \rk \eta_* \leq l + \rk j_*. 
\end{equation}
By Theorem \ref{JXPD}, there exists some isomorphism
$\redH^m (IX) \to \redH_m (IX)$ such that
\begin{equation} \label{equ.etazetaduality}
\xymatrix{
\redH^m (IX) \ar[r]^{\eta^*} \ar[d]_{\cong} & H^m (M) \ar[d]_{\cong}^{D_M} \\
\redH_m (IX) \ar[r]^{\zeta_*} & H_m (M,\partial M)
} \end{equation}
commutes. Therefore,
\[ \rk \zeta_* = \rk \eta^* = \rk \eta_*, \]
and by (\ref{equ.rketaleqlrkj}),
\[ \rk \zeta_* \leq l + \rk j_*. \]
The decomposition (\ref{equ.decomphixlw}) implies that
\[ \dim \redH_m (IX) = l + \dim W = l + \rk \zeta_* \leq 2l + \rk j_*.  \]
On the other hand, the decomposition $\redH_m (IX)=L\oplus V\oplus L'$ implies
\[ \dim \redH_m (IX) = l + \dim V + \dim L' = l + \rk j_* + \dim L'. \]
It follows that
\[ l + \rk j_* + \dim L' \leq 2l + \rk j_* \]
and thus
\[ \dim L' \leq l. \]
\end{proof}
In summary then, we have constructed a certain basis
\begin{equation} \label{equ.basishmix}  
\{ u_1, \ldots, u_l, v_1, \ldots, v_r, \overline{w}_1,\ldots ,\overline{w}_l \} 
\end{equation}
for $\redH_m (IX) = L\oplus V\oplus L'$.

\begin{remark}
The above proof shows that $\rk \eta_* \leq l+\rk j_* = l+r.$ Thus the restriction of $\eta_*$
to the subspace $A\subset H_m (M)$ in the decomposition (\ref{equ.hommdecomp})
is zero, which implies that $A\subset \ker \eta_*$ and so $A=\{ 0 \}$. The decomposition of $H_m (M)$
is thus seen to be
\begin{equation} \label{equ.hmdecomp}
H_m (M) = \rat \langle i_* (\overline{u}_1),\ldots, i_* (\overline{u}_l) \rangle \oplus
  \ker \eta_* \oplus \rat \langle \overline{e}_1, \ldots, \overline{e}_r \rangle. 
\end{equation}
In particular,
\[ Z =  \ker \eta_* \oplus \rat \langle \overline{e}_1, \ldots, \overline{e}_r \rangle. \]
\end{remark}

Let
\[  \{ u^1, \ldots, u^l, v^1, \ldots, v^r, \overline{w}^1,\ldots ,\overline{w}^l \} \]
be the dual basis for $\redH^m (IX)$. Setting
\[ L^\dagger = \rat \langle u^1, \ldots, u^l \rangle,~
  V^\dagger = \rat \langle v^1, \ldots, v^r \rangle,~
  (L')^\dagger = \rat \langle \overline{w}^1, \ldots, \overline{w}^l \rangle, \]
we get a dual decomposition
\[ \redH^m (IX) = L^\dagger \oplus V^\dagger \oplus (L')^\dagger. \]
We define the duality map
\[ d_{IX}: \redH_m (IX) \longrightarrow \redH^m (IX) \]
on basis elements to be
\begin{align*}
d_{IX} (u_j) &:= \overline{w}^j, \\
d_{IX} (\overline{w}_j) &:= u^j, \\
d_{IX} (v_j) &:= \zeta^* d_M (\overline{e}_j).
\end{align*}
We shall now prove that $d_{IX}$ is an isomorphism.
\begin{lemma}
The image $d_{IX} (V)$ is contained in $V^\dagger$.
\end{lemma}
\begin{proof}
In terms of the dual basis, $d_{IX} (v_j)$ can be expressed as a linear combination
\[ d_{IX} (v_j) = \sum_p \pi_p u^p + \sum_q \epsilon_q v^q + \sum_i \lambda_i \overline{w}^i. \]
The coefficients $\pi_p$ are 
\[ \pi_p = \left( \zeta^* d_M (\overline{e}_j) \right) (u_p) 
  = d_M (\overline{e}_j)(\zeta_* u_p)=0, \]
since $u_p \in L = \ker \zeta_*$.
Using (\ref{equ.dualwz}) and $d_M (\overline{e}_j)\in d_M (Z) = \rat \langle z^1,\ldots, z^s \rangle,$
we find
\[ \lambda_i = \left( \zeta^* d_M (\overline{e}_j) \right) (\overline{w}_i) 
  = d_M (\overline{e}_j)(w_i)=0. \]
\end{proof}

\begin{lemma} \label{lem.dixvinj}
The restriction $d_{IX}|:V \to V^\dagger$ is injective.
\end{lemma}
\begin{proof}
Suppose that $v=\sum_q \epsilon_q v_q$ is any vector $v\in V$ with $d_{IX} (v)=0$.
Then
\begin{align*}
0 = \eta^* d_{IX} (v) 
&= \eta^* \sum \epsilon_q d_{IX} (v_q) = \eta^* \sum \epsilon_q \zeta^* d_M (\overline{e}_q) \\
&= j^* d_M \sum \epsilon_q \overline{e}_q = d'_M \sum \epsilon_q j_* (\overline{e}_q) \\
&= d'_M \sum \epsilon_q e_q.
\end{align*}
Since $d'_M$ is an isomorphism, $\sum \epsilon_q e_q =0$ and by the linear independence of
the $e_q$, the coefficients $\epsilon_q$ all vanish. This shows that $v=0$.
\end{proof}
By definition, $d_{IX}$ maps $L$ isomorphically onto $(L')^\dagger$ and $L'$ isomorphically
onto $L^\dagger$. Since by Lemma \ref{lem.dixvinj}, $d_{IX}|:V\to V^\dagger$ is an isomorphism,
we conclude that the duality map $d_{IX}: \redH_m (IX)\to \redH^m (IX)$ is an isomorphism.

\begin{proposition} \label{prop.dixsymmetric}
The intersection form $\beta: \redH_m (IX)\times \redH_m (IX)\to \rat$ given by
$\beta (v,w) = d_{IX} (v)(w)$ is symmetric. In fact it is given in terms of the basis
(\ref{equ.basishmix}) by the matrix
\[ \begin{pmatrix}
0 & 0 & I \\
0 & S & 0 \\
I & 0 & 0
\end{pmatrix}, \]
where $I$ is the $l\times l$-identity matrix and
$S$ is a symmetric $r\times r$-matrix, representing the classical intersection form
on $\im j_*$ whose signature is the Novikov signature $\sigma (M,\partial M)$.
\end{proposition}
\begin{proof}
On $V$, we have
\begin{align*}
d_{IX} (v_i)(v_j)
&= \zeta^* d_M (\overline{e}_i)(v_j) = \zeta^* d_M (\overline{e}_i)(\eta_* \overline{e}_j) \\
&= d_M (\overline{e}_i)(j_* \overline{e}_j) = d_M (\overline{e}_i)(e_j) = d_M (\overline{e}_j)(e_i) \\
&= d_M (\overline{e}_j)(j_* \overline{e}_i) = \zeta^* d_M (\overline{e}_j)(\eta_* \overline{e}_i) \\
& = \zeta^* d_M (\overline{e}_j)(v_i) = d_{IX} (v_j)(v_i).
\end{align*}
These are the symmetric entries of $S$.
Between $V$ and $L$ we find
\[ d_{IX} (v_i)(u_j) =  \zeta^* d_M (\overline{e}_i)(u_j)= d_M (\overline{e}_i)(\zeta_* u_j)=0, \]
as $u_j \in L=\ker \zeta_*$. This agrees with
\[ d_{IX} (u_j)(v_i) = \overline{w}^j (v_i) =0, \]
by definition of the dual basis.
The intersection pairing between $V$ and $L'$ is trivial as well:
\[ d_{IX} (v_i)(\overline{w}_j) = \zeta^* d_M (\overline{e}_i)(\overline{w}_j) =
   d_M (\overline{e}_j)(\zeta_* \overline{w}_j) = d_M (\overline{e}_i)(w_j) =0, \]
since $d_M (\overline{e}_i) \subset d_M (Z)$.
This agrees with
\[ d_{IX} (\overline{w}_j)(v_i) = u^j (v_i) =0, \]
again by definition of the dual basis. On L,
\[ d_{IX} (u_i)(u_j) = \overline{w}^i (u_j)=0 \]
and on $L',$
\[ d_{IX} (\overline{w}_i)(\overline{w}_j) = u^i (\overline{w}_j) =0. \]
Finally, the intersection pairing between $L$ and $L'$ is given by
\[ d_{IX} (u_i)(\overline{w}_j) = \overline{w}^i (\overline{w}_j) = \delta_{ij} =
   u^j (u_i) = d_{IX} (\overline{w}_j)(u_i). \]
\end{proof}

Theorem \ref{HIsignature} follows readily from this proposition because
\[ \sigma (IX) = \sigma (S) + \sigma \begin{pmatrix} 0 & I \\ I & 0 \end{pmatrix} =
  \sigma (S) = \sigma (M,\partial M). \]

It remains to prove that both
\begin{equation} \label{equ.diazetadixetadmp}
\xymatrix{
\redH_m (IX) \ar[r]^{\zeta_*} \ar[d]_{d_{IX}} & H_m (M,\partial M) \ar[d]^{d'_M} \\
\redH^m (IX) \ar[r]_{\eta^*} & H^m (M)
} 
\end{equation}
and
\begin{equation} \label{equ.dianudixdkkdelta}
\xymatrix{
\redH_m (Q_{\geq k} E) \ar[r]^{\nu_*} & \redH_m (IX) \ar[d]^{d_{IX}} \\
H^{m-1} (\ftr_{<k} E) \ar[u]^{D_{k,k}} \ar[r]^{\delta^*} & \redH^m (IX)
} 
\end{equation}
commute.
We begin with diagram (\ref{equ.diazetadixetadmp}) and check the commutativity on basis elements. \\

1. We verify that $\eta^* d_{IX} (u_j) = d'_M \zeta_* (u_j)$ for all $j$. By exactness,
  $\zeta_* \eta_* i_* = j_* i_* =0$ and hence
 \[ d'_M \zeta_* (u_j) = d'_M \zeta_* \eta_* i_* (\overline{u}_j)=0. \]
So it remains to show that $\eta^* d_{IX} (u_j)=0$. We break this into three steps according to the
decomposition (\ref{equ.hmdecomp}). Evaluating on elements of the form $i_* \overline{u}_i$
yields
\[ \eta^* d_{IX} (u_j)(i_* \overline{u}_i) = (\eta^* \overline{w}^j)(i_* \overline{u}_i)
  = \overline{w}^j (\eta_* i_* \overline{u}_i) = \overline{w}^j (u_i)=0. \]
If $a$ is any element in $\ker \eta_*,$ then
\[ (\eta^* \overline{w}^j)(a) = \overline{w}^j (\eta_* a)=0. \]
Before evaluating on elements $\overline{e}_i,$ we observe that since
$\eta_* \overline{e}_i \in W$ (by (\ref{equ.etaoveiinw})) and
\[ \zeta_* (\eta_* \overline{e}_i) = j_* \overline{e}_i \in \im j_*, \]
we have
\[ \eta_* \overline{e}_i \in W \cap \zeta^{-1}_* (\im j_*) =V. \]
It follows that
\[ (\eta^* \overline{w}^j)(\overline{e}_i) = \overline{w}^j (\eta_* \overline{e}_i)=0. \]
Thus $\eta^* d_{IX} (u_j)=0$ as claimed.\\

2. On basis elements $v_j,$ the commutativity is demonstrated by the calculation
\begin{align*}
\eta^* d_{IX} (v_j)
&= \eta^* \zeta^* d_M (\overline{e}_j) = j^* d_M (\overline{e}_j) =
  d'_M j_* (\overline{e}_j) \\
&= d'_M \zeta_* \eta_* (\overline{e}_j) = d'_M \zeta_* (v_j).
\end{align*}

3. We prove that $\eta^* d_{IX} (\overline{w}_j) = d'_M \zeta_* (\overline{w}_j)$ for all $j$.
Again it is necessary to break this into three steps according to the
decomposition (\ref{equ.hmdecomp}). Evaluating on elements of the form $i_* \overline{u}_i$
yields
\[ \eta^* d_{IX} (\overline{w}_j)(i_* \overline{u}_i) = \eta^* (u^j)(i_* \overline{u}_i)
 = u^j (\eta_* i_* \overline{u}_i) = u^j (u_i) = \delta_{ij}
\]
and
\[ d'_M \zeta_* (\overline{w}_j)(i_* \overline{u}_i) = d'_M (w_j)(i_* \overline{u}_i)
 = d_M (i_* \overline{u}_i)(w_j) = w^i (w_j) = \delta_{ij}. \]
If $a$ is any element in $\ker \eta_*,$ then
\[ \eta^* (u^j)(a) = u^j (\eta_* a)=0=d_M (a)(w_j)= d'_M (w_j)(a), \]
using (\ref{equ.dualwz}) and $d_M (a) \in d_M (Z)$.
Finally, on elements $\overline{e}_i$ we find
\[ \eta^* (u^j)(\overline{e}_i) = u^j (\eta_* \overline{e}_i) = u^j (v_i)=0
 = d_M (\overline{e}_i)(w_j)= d'_M (w_j)(\overline{e}_i), \]
using (\ref{equ.dualwz}) and $d_M (\overline{e}_i) \in d_M (Z)$.
The commutativity of (\ref{equ.diazetadixetadmp}) is now established.\\

If $a\in H_m (M)$ and $b\in \redH_m (IX)$ are any elements, then using (\ref{equ.diazetadixetadmp}),
\begin{align*}
\zeta^* d_M (a)(b)
&= d_M (a)(\zeta_* b) = d'_M (\zeta_* b)(a) = (\eta^* d_{IX} b)(a) \\
&= d_{IX} (b)(\eta_* a) = d_{IX} (\eta_* a)(b),
\end{align*}
where the last equation uses the symmetry of $d_{IX},$ Proposition \ref{prop.dixsymmetric}.
Hence the diagram
\begin{equation} \label{equ.etadmzetadix} 
\xymatrix{
H_m (M) \ar[r]^{\eta_*} \ar[d]_{d_M} & \redH_m (IX) \ar[d]^{d_{IX}} \\
H^m (M,\partial M) \ar[r]^{\zeta^*} & \redH^m (IX)
} \end{equation}
commutes as well.
The cohomology braid of the triple
\[ \xymatrix{
\ftr_{<k} E \ar[r]^{F_{<k}} \ar[rd]_{\tau} & \partial M \ar[d]^i \\
& M
} \]
contains the commutative square
\begin{equation} \label{equ.deltaflkzetadelta}
\xymatrix{
H^{m-1} (\partial M) \ar[r]^{\delta^*} \ar[d]_{F^*_{<k}} & H^m (M,\partial M) \ar[d]^{\zeta^*} \\
H^{m-1} (\ftr_{<k} E) \ar[r]^{\delta^*} & \redH^m (IX).
} \end{equation}
We are now in a position to prove the commutativity of (\ref{equ.dianudixdkkdelta}).

Let $a\in H^{m-1} (\ftr_{<k} E)$ be any element.
We must show that $d_{IX} \nu_* D_{k,k} (a)=\delta^* (a)$.
As $F^*_{<k}: H^{m-1} (\partial M)\to H^{m-1} (\ftr_{<k} E)$ is surjective
(Lemma \ref{lem.flkcgkexact}), there exists an $\overline{a} \in H^{m-1} (\partial M)$
with $a = F^*_{<k} (\overline{a}).$
By Propositions \ref{prop.dob0}, \ref{prop.obzerodisdkl}, $D_{k,k}$ is the unique isomorphism such that
\[ \xymatrix{ \rCH{\partial M}{m-1} \ar[r]^{\tr{F}{k}^*} \ar[d]_{D_{\partial M}}^{\cong} & 
  \rCH{\ft{E}{k}}{m-1} \ar[d]^{D_{k,k}}_{\cong} \\ 
 \rH{\partial M}{m} \ar[r]^{\ctr{C}{k *}} & \rrH{\ctr{Q}{k}E}{m} } \]
commutes. Therefore,
\[ D_{k,k} (a) = D_{k,k} F^*_{<k} (\overline{a}) = C_{\geq k*} D_{\partial M}
 (\overline{a}). \]
Then, by the lower middle square in Diagram (\ref{equ.ijetazetanuctauf}),
\[ \nu_* D_{k,k} (a) = \nu_* C_{\geq k*} D_{\partial M} (\overline{a})
  = \eta_* i_* D_{\partial M} (\overline{a}). \]
Applying $d_{IX}$ and using the commutative diagram (\ref{equ.etadmzetadix}), 
we arrive at
\[ d_{IX} \nu_* D_{k,k} (a) = d_{IX} \eta_* i_* D_{\partial M} (\overline{a})
  = \zeta^* d_M i_* D_{\partial M} (\overline{a}). \]
Now the commutative diagram
\[ \xymatrix{
H_m (\partial M) \ar[r]^{i_*} & H_m (M) \ar[d]^{d_M} \\
H^{m-1} (\partial M) \ar[u]_{D_{\partial M}} \ar[r]^{\delta^*} &
  H^m (M,\partial M)
} \]
shows that
\[ d_{IX} \nu_* D_{k,k} (a) = \zeta^* \delta^* (\overline{a}), \]
which by Diagram (\ref{equ.deltaflkzetadelta}) equals $\delta^* F^*_{<k} (\overline{a}) = \delta^* (a),$ as
was to be shown.

\section{Sphere Bundles, Symplectic Toric Manifolds}
\label{sec.sphsymptorman}

We discuss equivariant Moore approximations for linear sphere bundles and
for symplectic toric manifolds.

\begin{proposition}\label{eulercondition}
Let $\xi = (E,\pi,B)$ be an oriented real $n$-plane vector bundle 
over a closed, oriented, connected, $n$-dimensional base manifold $B$.
Let $S(\xi)$ be the associated sphere bundle and
let $e_\xi \in H^n (B;\intg)$ be the Euler class of $\xi$.  
Then $S(\xi)$ can be given a structure group which allows for a degree $k$ equivariant Moore approximation,
for some $0<k<n,$ if and only if $e_\xi =0$.
\end{proposition}

\begin{proof}
Assume that $S(\xi)$ can be given a structure group which allows for a degree $k$ 
equivariant Moore approximation for some $0<k<n.$
If the fiber dimension $n$ of the vector bundle is odd, then the Euler class has order two. Since
$H^n (B;\intg)\cong \intg$ is torsion free, $e_\xi =0$. Thus we may assume
that $n=2d$ is even. We form the double
\[ X^{4d} = DE \cup_{SE} DE, \]
where $DE$ is the total space of the disk bundle of $\xi$, and $SE = \partial DE$.
Then $X$ is a manifold, but we may view it as a 2-strata pseudomanifold $(X,B)$ by taking
$B\subset X$ to be the zero section in one of the two copies of $DE$ in $X$.
For this stratified space, $M=DE$, $\partial M =SE,$ and $\hat{M} = TE,$ the Thom-space
of $\xi$. Since the double of any manifold with boundary is nullbordant, the signature of $X$
vanishes, $\sigma_{IH} (X)=\sigma (X)=0$. 
Note that a degree $k$ equivariant Moore approximation to $S^{n-1},$ some $0<k<n,$
is in particular an equivariant Moore approximation of degree
$\lfloor \frac{1}{2} (\dim S^{n-1} +1) \rfloor = \lfloor \frac{n}{2} \rfloor$.
Thus by Corollary \ref{IHsignature},
\[ \sigma_{IH} (TE)  = \sigma_{IH} (X) =0. \]
The middle intersection homology of the Thom space of a vector bundle is given by
\[ IH_n (TE) \cong \im (H_n (DE)\to H_n (DE,SE)), \]
\cite[p. 77, Example 5.2.5.3]{KW}. By homotopy invariance $H_n (DE)\cong H_n (B)\cong \rat [B],$
and by the Thom isomorphism $H_n (DE,SE) \cong H_0 (B)\cong \rat.$
The intersection form on the, at most one-dimensional, image is determined by the self-intersection number
$[B]\cdot [B]$ of the fundamental class of $B$, which is precisely the Euler number.
Since $\sigma_{IH} (TE)=0$, this self-intersection number, and thus $e_\xi$, must vanish.
(Note that in this case, the map $H_n (DE)\to H_n (DE,SE)$ is the zero map and
$IH_n (TE)=0$, for $IH_n (TE)\cong \rat$ and $[B]\cdot [B]=0$ would contradict the
nondegeneracy of the intersection pairing.)

Conversely, if $e_{\xi}=0$, then \cite[Thm. 2.10, p. 137]{HI} asserts that $\xi$ has a 
nowhere vanishing section. This section induces a splitting $\xi \cong \xi' \oplus \underline{\real}^1,$
where $\xi'$ is an $(n-1)$-plane bundle and $\underline{\real}^1$ denotes the trivial line bundle
over $B$. This splitting reduces the structure group from $\SO (n)$ to
$\SO (1)\times \SO (n-1) = \{ 1 \} \times \SO (n-1)$. The action of this reduced structure group
on $S^{n-1}$ has two fixed points; let $p \in S^{n-1}$ be one of them. Then 
$\{ p \} \hookrightarrow S^{n-1}$ is an $\{ 1 \} \times \SO (n-1)$-equivariant Moore approximation
for every degree $0<k<n$.
\end{proof}

\begin{example}\label{Toric Manifold}
A symplectic toric manifold is a quadruple $\left(M,\omega,T^n,\mu\right)$, where $M$ is a $2n$-dimensional, 
compact, symplectic manifold with non-degenerate closed $2$-form $\omega$, there is an effective 
Hamiltonian action of the $n$-torus $T^n$ on $M$, and $\mu:M\rightarrow \R^n$ is a choice of moment 
map for this action.  There is a one-to-one correspondence between such $2n$-dimensional 
symplectic toric manifolds and so-called
Delzant polytopes in $\R^n$, \cite{DE}, given by the assignment
	$$\left(M,\omega,T^n,\mu\right) \mapsto \Delta_M := \mu\left(M\right).$$
Recall that a polytope in $\R^n$ is the convex hull of a finite number of points in $\R^n$.
Delzant polytopes in $\R^n$ have the property that each vertex has exactly $n$ edges adjacent to it and
for each vertex $p$, every edge adjacent to $p$ has the form
$\{ p+tu_i ~|~ T_i \geq t\geq 0 \}$ with $u_i \in \intg^n,$ and $u_1,\ldots, u_n$ constitute a
$\intg$-basis of $\intg^n$.

Section $3.3$ of \cite{dS} uses the Delzant polytope $\Delta_M$ to construct Morse functions on $M$ 
as follows:
Let $X\in \R^n$ be a vector whose components are independent over $\rat$.
Then $X$ is not parallel to any facet of $\Delta_M$ and the orthogonal projection
$\pi_X: \real^n \to \real$ onto the line spanned by $X$,
$\pi_X (Y) = \langle Y,X \rangle,$ is injective on the vertices of $\Delta_M$.
By composing the moment map 
$\mu$ with the projection $\pi_X,$  
one obtains a Morse function $f_X =\pi_X \circ \mu : M \rightarrow \R$, $f_X (q) = \langle \mu (q),X\rangle,$
whose critical points are precisely the fixed points of the $T^n$ action.
The images of the fixed points under the moment map are
the vertices of $\Delta_M$. Since the coadjoint action is trivial on a torus,
$T^n$ acts trivially on $\R^n$, and as $\mu$ is 
equivariant, it is thus constant on orbits.
Hence the level sets of $\pi_X\circ \mu$ are $T^n$-invariant. The index of a critical point 
$p$ is twice the number of edge vectors $u_i$ of $\Delta_M$ at $\mu (p)$ whose inner product 
with $X$ is negative,
$\langle u_i , X \rangle < 0$. In particular, the index is always even.
For $a\in \real,$ we set $M_{a} = f_X^{-1} (-\infty,a] \subset M$.

Suppose that one can choose $X$ in such a way that the critical points satisfy:
\begin{enumerate}
\item[(C)] For any two critical points $p,q$ of $f_X$, if the index of $p$ is larger than 
the index of $q$, then $f_X (p) > f_X (q)$.
\end{enumerate}

Then, since $f_X$ is Morse, for each critical value $a$ of $f_X$ the set 
$M_{a+\epsilon}$ is homotopy equivalent to a CW-complex with one cell 
attached for each critical point $p$ with $f_X (p) < a+\epsilon$.  
(Here $\epsilon >0$ has been chosen so small that there are no critical values of $f_X$
in $(a,a+\epsilon]$.) The dimension of the cell 
associated to $p$ is the index of $f_X$ at $p$.
Let $2i$ be the index of any critical point $p\in M_{a+\epsilon}$ with $f_X (p)=a$.
If $q\in M_{a+\epsilon}$ is an arbitrary critical point of $f_X$, then
$f_X (q)\leq f_X (p)=a$ and thus the index of $q$ is at most $2i$ by condition (C).
Thus $M_{a+\epsilon}$ contains all cells of $M$ that have dimension at most $2i$ and no other cells.
Since $M$ has only 
cells in even dimensions, the cellular chain complex of $M$ has zero differentials in all degrees.  
Thus, since $f_X$ is equivariant, $M_{a+\epsilon} \hookrightarrow M$ is a $T^n$-equivariant 
Moore approximation of degree 
$2i+1$ (and of degree $2i+2$), and is a smooth manifold with boundary.

A particular case of this is complex projective space
$(\cp^n, \omega_{\operatorname{FS}}, T^n, \mu)$, where $\omega_{\operatorname{FS}}$
is the Fubini-Study symplectic form and $T^n$ acts on $\cp^n$ by
\[ (e^{it_1}, \ldots, e^{it_n})\cdot (z_0: z_1: \cdots :z_n) =
     (z_0: e^{it_1} z_1: \cdots : e^{it_n} z_n). \]
On page $26$ of \cite{MI2}, an equivariant Morse function with $n+1$ critical points is constructed,
the $i$-th one having index $2i$ and critical value $i$.
Using this we obtain equivariant Moore approximations 
to $\cp^n$ of every degree with respect to the torus action.

In the case that $M$ is $4$-dimensional, condition (C) is satisfied. 
The Delzant polytope $\mu (M)$ associated to a $4$-dimensional symplectic toric manifold 
$(M,\omega,T^2,\mu)$ is a $2$-dimensional polytope in $\R^2$.  
As $M$ is compact, $f_X$ attains its minimum $m$ and its maximum $m'$ on $M$.
Let $p_{\min} \in M$ be a critical point with $f_X (p_{\min})=m$ and let
$p_{\max} \in M$ be a critical point with $f_X (p_{\max})=m'$.
Suppose that $p\in M$ is any critical point such that $f_X (p)=m$.
Then $\pi_X \mu (p)=m=\pi_X \mu (p_{\min})$. The moment images
$v=\mu (p)$ and $v_{\min} = \mu (p_{\min})$ are vertices of $\Delta_M$.
Since the projection $\pi_X$ is injective on vertices, we have $v=v_{\min}$.
Now as $\mu$ maps the fixed points (which are precisely the  critical points)
bijectively onto the vertices, it follows that $p=p_{\min}$. This shows
that $p_{\min}$ is unique and similarly $p_{\max}$ is unique.
The index of $p_{\min}$ is $0$, while the index of $p_{\max}$ is $4$.
Thus $\langle u_1, X \rangle \geq 0$ and $\langle u_2, X \rangle \geq 0$ at $v_{\min}$
and
$\langle u_1, X \rangle <0$ and $\langle u_2, X \rangle <0$ at $v_{\max}$.
Geometrically, this means that the two edges that go out from $v_{\min}$ point
in the same half-plane as $X$, while the outgoing edges at $v_{\max}$ point
in the half-plane complementary to the one of $X$.
If $v$ is any vertex of the moment polytope different from $v_{\min}, v_{\max}$,
then by the convexity of $\Delta_M,$ one of the two outgoing edges must point
in $X$'s half-plane, while the other outgoing edge points into the complementary half-plane,
yielding an index of $2$.
If $p\in M$ is a critical point different from $p_{\min}, p_{\max},$ then
$\mu (p)$ is a vertex different from $v_{\min}, v_{\max}$ and thus must have
index $2$. From this, it follows that condition (C) is indeed satisfied:
If $p,q$ are critical points such that $p$ has larger index than $q$, then there
are two cases: $p$ has index $4$ and $q$ has index in $\{0,2 \}$, or
$p$ has index $2$ and $q$ has index $0$. In the first case, $p=p_{\max}$ and
in the second case $q = p_{\min}$. In both cases it is then clear, using the uniqueness
of $p_{\min}, p_{\max},$ that 
$f_X (p) > f_X (q)$. We have thus shown:
\begin{proposition} \label{prop.toricmfddim4}
Every $4$-dimensional symplectic toric manifold $(M,\omega,T^n,\mu)$ 
has an equivariant Moore approximation $\tr{M}{k}$ of degree $k$ for every $k\in\Z$. 
Furthermore, the space $\tr{M}{k}$ can be chosen to be a smooth compact 
codimension $0$ submanifold-with-boundary of $M$.
\end{proposition}
\end{example}


\begin{thebibliography}{CHS}

\bibitem{BA} Banagl, Markus, {\em Intersection spaces, spatial homology truncation, and string theory},
   Lecture Notes in Math. No. 1997, Springer Verlag, 2010.	
	
\bibitem{BAGGD} M.~Banagl, {\em Isometric group actions and the cohomology of
          flat fiber bundles}, Groups Geom. Dyn. \textbf{7} (2013), no.~2, 293 -- 321.

\bibitem{Ba-new} M.~Banagl, {\em Foliated Stratified Spaces and a de Rham Complex Describing Intersection Space Cohomology}, preprint, arXiv:1102.4781, to appear, J. Diff. Geo.

\bibitem{banagl-tiss} M.~Banagl, {\em Topological invariants of stratified spaces}, Springer
  Monographs in Mathematics, Springer-Verlag Berlin Heidelberg, 2007.

\bibitem{bandepth2} M.~Banagl, {\em First cases of intersection spaces in stratification depth 2},
   J. Singularities \textbf{5} (2012), 57 -- 84.

\bibitem{bbm} M.~Banagl, N.~Budur, L.~Maxim, {\em Intersection Spaces, Perverse Sheaves
  and Type IIB String Theory}, Adv. Theor. Math. Physics \textbf{18} (2014), no.~2, 363 -- 399.

\bibitem{bahu} M.~Banagl, E.~Hunsicker, {\em Hodge Theory For Intersection Space Cohomology}, preprint.

\bibitem{bmjta} M.~Banagl, L.~Maxim, {\em Deformation of Singularities and the Homology
  of Intersection Spaces}, J. Topol. Anal. \textbf{4} (2012), no.~4.

\bibitem{bmjs} M.~Banagl, L.~Maxim, {\em Intersection Spaces and Hypersurface Singularities},
 J. Singularities \textbf{5} (2012), 48 -- 56.

\bibitem{broeckerjaenich} Br{\"o}cker, Theodor, and Klaus J{\"a}nich,
         {\em Einf{\"u}hrung in die Differentialtopologie},
         Heidelberger Taschenb{\"u}cher, Springer-Verlag Berlin Heidelberg, 1973. 

\bibitem{carlsson} G.~Carlsson, {\em A counterexample to a conjecture of Steenrod},
  Invent. Math. \textbf{64} (1981), 171 -- 174.

\bibitem{cheeger2}
J.~Cheeger, {\em On the spectral geometry of spaces with cone-like
  singularities}, Proc. Natl. Acad. Sci. USA \textbf{76} (1979), 2103 -- 2106.

\bibitem{cheeger1}
J.~Cheeger, {\em On the {H}odge theory of {R}iemannian pseudomanifolds}, Proc.
  Sympos. Pure Math. \textbf{36} (1980), 91--146.

\bibitem{cheeger3}
J.~Cheeger, {\em Spectral geometry of singular {R}iemannian spaces}, J.
  Differential Geom. \textbf{18} (1983), 575 -- 657.

\bibitem{dS} A.~C.~da Silva, {\em Symplectic toric manifolds,} 
  Symplectic Geometry of Integrable Hamiltonian Systems, Birkh{\"a}user, 2003.
	
\bibitem{DE} T.~Delzant, 
 {\em Hamiltoniens p{\'e}riodiques et images convexes de l'application moment.} 
  Bulletin de la Soci{\'e}t{\'e} math{\'e}matique de France \textbf{116} (1988), no.~3, 315 -- 339.
	
\bibitem{dundas} Dundas, Bj{\o}rn Ian, {\em Differential Topology,} preprint available at
      http://folk.uib.no/nmabd/dt/dtcurrent.pdf.

\bibitem{FM} G.~Friedman, J.~E.~McClure, {\em Cup and cap products in intersection (co) homology,} 
  Advances in Math. \textbf{240} (2013), 383 -- 426.
	
\bibitem{gm1}
M.~Goresky, R.~D.~MacPherson, {\em Intersection homology theory}, Topology
  \textbf{19} (1980), 135 -- 162.

\bibitem{gm2}
M.~Goresky and R.~D.~MacPherson, {\em Intersection homology {II}}, 
 Invent. Math. \textbf{71} (1983), 77 -- 129.

\bibitem{HI} Hirsch, Morris W., {\em Differential topology}, vol. 33d, Springer-Verlag, 1976.
	
\bibitem{IL} S.~Illman, {\em The equivariant triangulation theorem for actions of compact Lie groups,}
    Math. Ann. \textbf{262} (1983), 487 -- 501.

\bibitem{K} H.~C.~King, {\em Topological invariance of intersection homology without sheaves.} 
  Topology and its Applications \textbf{20} (1985), no.~2, 149 -- 160.
	
\bibitem{KW} Kirwan, Frances, and Woolf, Jonathan, 
  {\em An introduction to intersection homology theory}, CRC Press, 2006.
	
\bibitem{mathertopstab} J.~Mather, {\em Notes on topological stability,}
         Bulletin of the Amer. Math. Soc. \textbf{49} (2012), no.~4, 475 -- 506.

\bibitem{MI2} Milnor, John Willard, {\em Morse theory}, no.~51, Princeton university press, 1963.

\bibitem{MS} Milnor, John W., and Stasheff, James D., 
 {\em Characteristic classes,} vol.~76, Annals of Math. Studies, 1974.
	
\bibitem{pflaum} Pflaum, Markus, {\em Analytic and Geometric Study of Stratified Spaces,} 
          Lecture Notes in Math. \textbf{1768}, Springer Verlag, 2001.

\bibitem{SI} P.~H.~Siegel, {\em Witt spaces: a geometric cycle theory for KO-homology at odd primes,} 
  Amer. J. of Math. \textbf{105} (1983), no.~5, 1067 -- 1105.	

\bibitem{SP} Spiegel, M. {\em K-Theory of Intersection Spaces}, PhD thesis,
  Ruprecht-Karls-Universit\"at Heidelberg, 2013.
\end{thebibliography}
\end{document}